\documentclass[12pt,reqno]{amsart}
\usepackage{times}
\usepackage{amssymb}
\usepackage{fullpage} 
\usepackage{setspace}
\usepackage{mathtools}
\usepackage{comment}
\usepackage{color}
\usepackage{amsmath,amsfonts,amsthm,amssymb,amscd,url}
\usepackage{graphicx}
\usepackage{mathrsfs}
\usepackage{todonotes}
 \usepackage{dsfont}
\usepackage{bbold}

\usepackage[parfill]{parskip} 

\usepackage[toc,page]{appendix}

\usepackage{cite}
\usepackage[colorlinks=true, linkcolor=blue, citecolor=green]{hyperref}

\newcommand{\Z}{\mathcal{Z}}

\newcommand{\I}{\mathcal{I}}
\newcommand{\J}{\mathcal{J}}

\newtheorem{theorem}{Theorem}[section]
\newtheorem{conjecture}[theorem]{Conjecture}
\newtheorem*{conjecture*}{Conjecture}
\newtheorem{lemma}{Lemma}[section]
\newtheorem{prop}{Proposition}[section]
\newtheorem*{remark*}{Remark}
\newtheorem{remark}[theorem]{Remark}
\newtheorem{corollary}[theorem]{Corollary}
\theoremstyle{definition}

\newtheorem*{definition*}{Definition}

\renewcommand{\mod}[1]{{\ifmmode\text{\rm\ (mod~$#1$)}\else\discretionary{}{}{\hbox{ }}\rm(mod~$#1$)\fi}}

\newcommand{\C}{{\mathbb C}}

\newcommand{\N}{{\mathbb N}}
\newcommand{\R}{{\mathbb R}}

\newcommand{\K}{\mathcal{K}}
\newcommand{\Ll}{\mathcal{L}}

\newcommand{\m}{{\bf m}}

\newcommand{\e}{\varepsilon}

\allowdisplaybreaks

\begin{document}

\newcommand{\nathan}[1]{{\color{purple} \sf  Nathan: #1}}
\newcommand{\alia}[1]{{\color{red} \sf  Alia: #1}}
\newcommand{\fatma}[1]{{\color{blue} \sf  Fatma: #1}}

\title[Mean values of long Dirichlet polynomials]{Mean values of long  Dirichlet polynomials \\ with divisor coefficients}

\thanks{Alia Hamieh and Nathan Ng were supported by the NSERC discovery grants RGPIN-2018-06313 and RGPIN-2020-06032, respectively. Fatma \c{C}i\c{c}ek was supported by a Pacific Institute for the Mathematical Sciences (PIMS) postdoctoral fellowship at the University of Northern British Columbia. This research was also funded by the PIMS Collaborative Research Group {\it $L$-functions in Analytic Number Theory}.}




\author[F. \c{C}\.{i}\c{c}ek, A. Hamieh and N. Ng]{Fatma \c{C}\.{i}\c{c}ek, Alia Hamieh and Nathan Ng}
\address{University of Northern British Columbia\\ Department of Mathematics and Statistics \\ 3333 University Way\\ Prince George, BC\ V2N4Z9\\ Canada}
\email{alia.hamieh@unbc.ca}
\address{University of Northern British Columbia\\ Department of Mathematics and Statistics \\ 3333 University Way\\ Prince George, BC\ V2N4Z9\\ Canada}
\email{cicek@unbc.ca}
\address{University of Lethbridge \\ Department of Mathematics and Computer Science \\ 4401 University Drive \\ Lethbridge, AB \ T1K 3M4 \\ Canada}
\email{nathan.ng@uleth.ca}


\begin{abstract}
In this article, we prove an asymptotic formula for the mean value of long smoothed Dirichlet polynomials with divisor coefficients. Our result has a main term that includes all lower order terms and a power saving error term. This is derived from a more general theorem on mean values of long smoothed Dirichlet polynomials that was previously established by the second and third authors in \cite{HN}. We thus establish a stronger form of a conjecture of Conrey and Gonek~\cite[Conjecture 4]{CG} in the case of divisor functions.  
\end{abstract}

\maketitle


\section{Introduction}

Mean values of Dirichlet polynomials play an important role in analytic number theory.  
They have important applications to zero-density estimates, primes in short intervals, gaps between primes and mean values of $L$-functions.
Although we will describe some elements of the theory, one may consult \cite[Ch. 9,10]{IK} and \cite[Ch. 7]{M} for a comprehensive discussion on mean values of Dirichlet polynomials. 

For a sequence of complex numbers $( a(n) )$, an associated Dirichlet polynomial is a partial sum in the form
\begin{equation*}
 \sum_{n\leq K} \frac{ a(n)}{ n^{s}}.
\end{equation*}
By \cite[Corollary 3]{MV}, this has the approximate behavior
\begin{equation}
   \label{diagestimate}
\frac{1}{T}   \int_{0}^{T}  \Big| \sum_{n\leq K} a(n) n^{-\sigma-it}  \Big|^2 \, dt 
  \asymp   \sum_{n\leq K} |a(n)|^2 n^{-2\sigma} \quad \text{as} \quad K\to \infty,
\end{equation}
provided that $K= O(T)$. If $K=o(T)$, then $\asymp$ can be replaced by $\sim$ and this is an asymptotic formula. 


Note that the integral on the left-hand side of \eqref{diagestimate} is called a mean value of the Dirichlet polynomial. If $K$ is not $O(T)$, then this integral is referred to as a mean value of a long Dirichlet polynomial and it is considerably more difficult to evaluate. Observe that when the left-hand side of \eqref{diagestimate} is expanded out via the identity $|z|^2 = z \overline{z}$, one encounters correlation sums in the form
\begin{equation}
 \label{correlation}
 \sum_{n \le x} a(n) \overline{a(n+h)} \quad \text{for} \quad h\in \mathbb{Z}^+,
\end{equation}
which are viewed as part of the off-diagonal contribution. In this case, the integral in \eqref{diagestimate} depends, in a crucial way, on the asymptotic behaviour of such correlation sums. The work of Goldston and Gonek \cite{GG} provides very precise formulae for mean values of this type under some conditions on the behavior of $(a(n))$. Indeed, their work can lead to asymptotic formulae for mean values of general Dirichlet polynomials in the case that $T \le K \le T^{1+\eta}$ for some $\eta < 1$ if there is square-root cancellation in the error term of their formula for \eqref{correlation}. The reader is referred to Theorems 1-3 and their corollaries in \cite{GG}. 

Later, in their work on the sixth and eighth moments of the Riemann zeta function \cite{CG}, Conrey and Gonek conjectured an asymptotic formula for the mean values of long Dirichlet polynomials when $a(n)=\tau_k(n)$ and $K=T^{1+\eta}$ with $0<\eta<1$. Here for $k \in \mathbb{N}$, $\tau_k$ denotes the $k$-th divisor function, which is defined as 
    \[
     \tau_k(n) = \# \{ (n_1, \ldots, n_k) \in \mathbb{N}^k \ | \ n_1 \cdots n_k = n \}
    \quad  \text{for }\,\, n \in \mathbb{Z}^+.
    \]
For example, for $k=2$, $ \tau_2(n)$ is the ordinary divisor function $d(n)$.

 \begin{conjecture}[Conrey-Gonek]{\cite[Conjecture 4]{CG}} \label{conjCG}
    Let $T$ be sufficiently large and $K=T^{1+\eta}$ with $\eta \in (0,1)$. Then we have
    \begin{equation*}
      \label{meanvaluekk}
        \int_{T}^{2T} \bigg | \sum_{n\leq K}\frac{\tau_{k}(n)}{n^{\frac12+it}}\bigg|^2\;dt \sim \frac{a_k}{\Gamma(k^2+1)} \, w_k \left(\frac{\log K}{\log T} \right)T (\log T)^{k^2},
    \end{equation*}
    where
    \begin{align*}
    a_k&=\prod_{p}\left\{\left(1-\frac{1}{p}\right)^{k^2}\sum_{\alpha=0}^{\infty}\frac{\tau_k^2(p^\alpha)}{p^\alpha}\right\},
    \\
     \hspace{1.8 cm}  
     w_k(x) &= x^{k^2}  \bigg\{ 1- \sum_{n=0}^{k^2-1} \binom{k^2}{n+1}\gamma_k(n) (-1)^n (1-x^{-    n-1})  \bigg\},
    \end{align*}
    and 
    \begin{equation*}  \label{gammakn}
      \gamma_k(n) =  \sum_{i=1}^{k} \sum_{j=1}^{k}  \binom{k}{i} \binom{k}{j}
       \binom{n-1}{i+j-2 }\binom{i+j-2}{j-1}  \quad \text{for } \, n\in\mathbb{Z}^+ \,\, \text{and } \, \gamma_k(0)=k. 
    \end{equation*}
\end{conjecture}

 The  case $k=2$ of Conjecture \ref{conjCG} was established by Bettin and Conrey \cite{BC} for all $\eta>0$. In this article we prove a stronger form of the conjecture in the same case, but for $0<\eta <\frac13$ and for smoothed Dirichlet polynomials. To be precise, we obtain all lower order terms with a power savings error term. We note that both the error term and the range for $\eta$ in our theorem below depend directly on bounds for the error term in the binary additive divisor problem. We discuss this in more detail in Remark \ref{rmk:thm} below.

Before presenting our result, we need to set some notation. Let $(a(n))$ and $(b(n))$ be sequences and $\varphi$ be some real-valued smooth function. We will specify the properties that $\varphi$ is required to have in Section \ref{ss-phi}. We define the smoothed Dirichlet polynomials 
  \begin{equation*}
  \label{ABphi} 
  \mathds{A}_{a,\varphi}(s) = \sum_{n =1}^{\infty} \frac{a(n)}{n^s} \varphi \Big( \frac{n}{K} \Big)
  \quad
  \text{and}
  \quad
  \mathds{B}_{b,\varphi}(s) = \sum_{n =1}^{\infty} \frac{b(n)}{n^s} \varphi \Big( \frac{n}{K} \Big).
   \end{equation*}
We then consider the mean value
   \begin{equation} \label{DabwK}
    \mathscr{D}_{a,b;  \omega}(K) = \int_{\mathbb{R}} \omega(t)  \mathds{A}_{a,\varphi}(\tfrac{1}{2}+it) 
    \mathds{B}_{b,\varphi}(\tfrac{1}{2}-it)\;dt, 
    \end{equation}
where $\omega$ is a complex-valued smooth function that satisfies the conditions
 \begin{align}
     & \label{cond1}  \bullet \omega \text{ is smooth}. \\
    & \label{cond2} \bullet \text{The support of }  \omega  \text{ lies in } [c_1T,c_2T]  \text{ where } 0 < c_1 < c_2.\\
    & \label{cond3}  \bullet  \text{For some positive absolute constant $\nu$, there exists } T_0  \ge T^{\nu} 
    \text{ such that }  
    \\
    &T_0 \ll T \text{ and }
      \omega^{(j)}(t) \ll T_{0}^{-j}.\notag
    \end{align}
The Fourier transform of $\omega$ is 
    \begin{equation}
      \label{ft}
      \widehat{\omega}(u) = \int_{\mathbb{R}} \omega(t) e^{-2 \pi i ut} dt.
    \end{equation}
It satisfies the following property.  \begin{equation}
      \label{whatbound}
   \text{ If } \,  |u| \gg T_{0}^{-1+\e},  \text{ then } \,
     | \widehat{\omega}(u)|  \ll T^{-A} \text{ for any } A >0.
    \end{equation}
Since throughout the paper we will only study the case where $a(n)=\tau_k(n)$ and  $b(n)=\tau_\ell(n)$ for some positive integers $k, \ell$, in order to simplify our notation we set 
$$ 
\mathscr{D}_{k,\ell;\omega}(K) :=\mathscr{D}_{\tau_k,\tau_{\ell};\omega}(K).
$$
We also need to introduce some real sequences $(g_j)$ and $(\delta_j)$. These are defined as coefficients in the following Taylor series.  
\begin{align*}
     f(s) & := s\zeta(1+s)  = \sum_{j=0}^{\infty} g_j s^j,    \\
     h(s) & :=   \frac{1}{\zeta(2+s)}  = \sum_{j=0}^{\infty} \delta_j s^j.   
\end{align*}
Another sequence $(c_j)$, which depends on the smoothing function $\varphi$, is defined as follows. Let 
\begin{equation*}
  G(s) :=  -2 \int_{0}^{\infty}  \varphi(t) \varphi'(t) 
    t^s \, dt.
\end{equation*}
Observe that $G(s)$ is entire. We then write its Taylor series expansion as
\begin{equation*}
  G(s)  =  \sum_{j=0}^{\infty} c_j s^j.
\end{equation*} 

With these definitions in hand, we can state our main result. 
\begin{theorem} \label{thm22}
    Let $K=T^{1+\eta} $ with $ 0 < \eta<\frac13$. Suppose that a weight function $\omega$ satisfies conditions \eqref{cond1}, \eqref{cond2} and \eqref{cond3} with $ \, \displaystyle \nu>\tfrac{5+3(\eta+1)}{9}$, while $\varphi$ is a function satisfying the conditions \eqref{eq:cond varphi}. Then for
    \[
\mathscr{D}_{2,2;\omega}(K) 
= \int_{\mathbb R} \omega(t) \,
\bigg| \sum_{n=1}^\infty 
\frac{\tau_2(n)}{n^{1/2+it}} \varphi\left( \frac{n}{K}\right)\bigg|^2 \mathop{dt},
\]
we have  
        \begin{equation*}
              \mathscr{D}_{2,2;\omega}(K) = 
             \sum_{j=0}^{4}  \int_{-\infty}^{\infty} \omega(t)  Q_j\left(\log K, \log  \frac{t}{2 \pi}  \right) \, dt 
              + O\bigg(T^{\tfrac34(1+\eta)+\e}\Big(\frac{T}{T_0}\Big)^{\frac94}+T^{1-\tfrac{\eta}{2}}\bigg) ,
         \end{equation*}
 where each $Q_j(x,y) \in \mathbb{R}[x,y]$ is a polynomial of degree $j$ given by 
         \begin{equation*}
           Q_4(x,y) =  \frac{1}{4!\zeta(2)}
          \big(-x^4 +8x^3y-24x^2 y^2 +32 xy^3 -14 y^4\big),
         \end{equation*}
\begin{equation*}
\begin{split}
\label{Q3}
Q_3(x,y)
 & =\left(\frac{2\delta_0g_1}{3}+\frac{\delta_1}{3}-\frac{c_1\delta_0}{6} \right)x^3+\big(-4\delta_0g_1-2\delta_1+c_1\delta_0\big)x^2y\\
 &+\big(8\delta_0g_1+4\delta_1-2c_1\delta_0\big)xy^2+\left(-4\delta_0g_1-2\delta_1+\frac{4c_1\delta_0}{3}\right)y^3,
\end{split}
\end{equation*}

\begin{equation*}
\begin{split}
\label{Q2}
Q_2(x,y)
=& \, \left(-2\delta_0g_2-3\delta_0g_1^2-4\delta_1g_1-2\delta_2+2c_1\delta_0g_1+c_1\delta_1
-\frac{c_2\delta_0}{2} \right)x^2\\
&+\big(8\delta_0g_2+12\delta_0g_1^2+16\delta_1g_1+8\delta_2-8c_1\delta_0g_1-4c_1\delta_1+2c_2\delta_0\big)xy\\
&+\big(-5\delta_0g_1^2-4\delta_2-6\delta_0g_2-8\delta_1g_1+8c_1\delta_0g_1+4c_1\delta_1-2c_2\delta_0\big)y^2,
\end{split}
\end{equation*}

\begin{equation*}
\begin{split}
 \label{Q1}
Q_1(x,y)=& \, \big(4\delta_0g_3+12\delta_0g_1g_2+4\delta_0g_1^3+8\delta_1g_2+12\delta_1g_1^2+16\delta_2g_1+8\delta_3 \\
&-4c_1\delta_0g_2-6c_1\delta_0g_1^2-8c_1\delta_1g_1-4c_1\delta_2+4c_2\delta_0g_1+2c_2\delta_1-c_3\delta_0\big)x\\&+\big(-12\delta_0g_3-4\delta_0g_1g_2-8\delta_1g_2+4\delta_1g_1^2+4\delta_0g_1^3\\
&+8c_1\delta_0g_2+12c_1\delta_0g_1^2+16c_1\delta_1g_1+8c_1\delta_2-8c_2\delta_0g_1-4c_2\delta_1+2c_3\delta_0\big)y,
\end{split}
\end{equation*}
and
\begin{equation*}
\begin{split}
 \label{Q0}
Q_0(x,y)   
=& \, 16 \delta_{4}-16 \delta_{1} g_{3}+32 \delta_{3} g_{1}+32 g_{1}^{2} \delta_{2}-24 \delta_{0} g_{4}+8 g_{2}^{2} \delta_{0}+5 \delta_{0} g_{1}^{4}
+16 \delta_{1} g_{1}^{3} -8 \delta_{0} g_{1} g_{3} \\
&+16 \delta_{1} g_{1} g_{2}+12 \delta_{0} g_{1}^{2} g_{2}+12 g_{1}^{2} \delta_{1} c_{1} +12 \delta_{0} g_{1} g_{2}c_1+8 \delta_{3}c_1+4 \delta_{0} g_{1}^{3}c_1 
+4 \delta_{0} g_{3}c_1 \\
+8 \delta_{1}& g_{2} c_{1} +16 g_{1} \delta_{2}c_1 
 -4 \delta_{2}c_2-6 g_{1}^{2} \delta_{0} c_{2}-4 \delta_{0} g_{2}c_2-8 g_{1} \delta_{1}c_2+4 g_{1} \delta_{0} c_{3}+2 \delta_{1} c_{3}-\delta_{0} c_{4}.
 \end{split}
\end{equation*}
\end{theorem}
In Appendix \ref{proof of cor}, we show how to remove the smooth function $\omega$ and derive the following result.

\begin{corollary}\label{cor}
    Let $K=T^{1+\eta} $ with $ 0 < \eta<\frac13$, and let $\varphi$ be a function satisfying the conditions in \eqref{eq:cond varphi}. Then, as $T\to\infty$, we have
   \begin{equation*}
 \int_{T}^{2T}
\bigg| \sum_{n=1}^\infty 
\frac{\tau_2(n)}{n^{1/2+it}} \varphi\left( \frac{n}{K}\right)\bigg|^2 \mathop{dt}
  =      \sum_{j=0}^{4}  \int_{T}^{2T}  Q_j\Big(\log K, \log \frac{t}{2 \pi} \Big) \, dt 
                +O\Big(T^{\max\left\{\tfrac{12+3\eta}{13} ,1-\tfrac{\eta}{2}\right\}}\Big),
    \end{equation*}
where the polynomials $Q_i(x,y)$ are as given in Theorem \ref{thm22}.
\end{corollary}

Observe that asymptotically, this result has the same leading term as the one in the conjecture of Conrey and Gonek in the case $k=2$ for $0 < \eta < \frac{1}{3}$.


\begin{remark}\label{rmk:thm}
 Our theorem depends on  a result of Hughes and Young~\cite[Theorem 5.1 and (74)]{HY}, where they applied Duke, Friedlander and Iwaniec's version of the $\delta$-method (see \cite{DFI}). Their work only makes use of the Weil bound for Kloosterman sums. In the recent articles of Aryan~\cite{Ar} and Topacogullari~\cite{T1}, the main theorem in  \cite{DFI} is improved by applying the spectral theory of automorphic forms
and bounds for sums of Kloosterman sums. By applying these more advanced techniques, it is likely that both the error term and the range of $\eta$  in our theorem can be improved. 

In fact, if  the binary divisor conjecture $\mathcal{AD}_{2,2}(\vartheta_{2,2},C_{2,2},\beta_{2,2})$ holds for a triple $(\vartheta_{2,2}, C_{2,2},\beta_{2,2}) \in \big[\tfrac{1}{2}, 1\big) \times [0,\infty) \times (0,1]$ (see Conjecture \ref{conj:shifted-sum} below for notation), then Theorem \ref{thm22} holds for $\eta<\frac{1}{\vartheta_{2,2}}-1$ and $\nu>\frac{C_{2,2}+(\vartheta_{2,2}+\varepsilon)(\eta+1)}{1+C_{2,2}}$ with an error term $O\bigg(T^{\vartheta_{2,2}(1+\eta)+\e}\Big(\frac{T}{T_0}\Big)^{1+C_{2,2}}+T^{1-\tfrac{\eta}{2}}\bigg) $.  In particular, if $\vartheta_{2,2}= \frac{1}{2}$, then Theorem \ref{thm22} holds for Dirichlet polynomials with length $K=T^{c}$ for any $c <2$. 
\end{remark}

 Our approach in proving Theorem \ref{thm22} is slightly different from those of Goldston and Gonek \cite{GG} and Conrey and Gonek \cite{CG}. In both works, one of the key steps is to express the mean value in \eqref{diagestimate} in terms of the correlation sums in \eqref{correlation} via partial summation. Whereas in the earlier work of the second and the third authors \cite{HN}, the starting point is to split the sum into sums over dyadic intervals via a smooth partition of unity.  Furthermore, they also work with  shifted divisor functions. Conditionally on the additive divisor conjecture~\cite[Conjecture~4]{HN}, they compute the mean value
\begin{equation*}
 \mathscr{D}_{ \sigma_{\I}, \sigma_{\J};\omega}(K),
\end{equation*}
where 
\[
    \sigma_{\I}(n) = \sum_{d_1 \cdots d_k=n} d_{1}^{-a_1} \cdots d_{k}^{-a_k}
 \quad \text{ and } \quad
     \sigma_{\J}(n) = \sum_{d_1 \cdots d_{\ell}=n} d_{1}^{-b_1} \cdots d_{k}^{-b_{\ell}}
\]
are shifted divisor functions associated to sets of complex numbers $\I = \{ a_1, \ldots, a_k \}$ and $\J = \{ b_1, \ldots, b_{\ell} \}$. Then $ \mathscr{D}_{ \sigma_{\I}, \sigma_{\J};\omega}(K)$ is evaluated by using a smooth partition of unity. Thus, instead of the correlation sums as in \eqref{correlation}, the authors work with the smoothed correlation sums  
\begin{equation}
  \label{smoothsum}
   \sum_{\substack{m,n\in \mathbb{Z} \\ m-n=h}}  \sigma_{\I}(n)  \sigma_{\J}(n)  F(m,n),
\end{equation}
where $F$ is a smooth function defined on a box $[M,2M] \times [N,2N]$. The main term for $ \mathscr{D}_{ \sigma_{\I}, \sigma_{\J};\omega}(K)$ is expressed in terms of a ``diagonal" contribution and an ``off-diagonal" contribution.  The diagonal contribution 
equals a contour integral involving the
Dirichlet series 
\[
  Z_{\I,\J}(s)=\sum_{m=1}^{\infty}\frac{\sigma_{\I}(m)\sigma_{\J}(m)}{m^{1+s}}.
\]
These contour integrals can be evaluated  similarly to integrals that one encounters in standard applications of Perron's formula.

The most difficult part is the computation of the ``off-diagonal" terms.  They 
may be expressed as a certain average of sums of type \eqref{smoothsum}. On the additive divisor conjecture, conjectural main terms for sums of this type are inserted and a formula for $ \mathscr{D}_{ \sigma_{\I}, \sigma_{\J};\omega}(K)$ is obtained. This idea of considering smoothed sums had originated in the work of \cite{DFI} and later was employed in a similar context in the articles~\cite{HY, Ng, NSW}.  Once the main terms from the additive divisor conjecture are inserted, there is still a lengthy calculation that needs to be done.  One encounters Dirichlet series of the shape 
\begin{equation}\label{H}
H_{\I,\J; \{ a_{i_1} \} , \{ b_{i_2} \}}(s)=
\sum_{r=1}^{\infty}\sum_{q=1}^{\infty} \frac{c_q(r)G_{\I}(1-a_{i_1},q)G_{\J}(1-b_{i_2},q)  }{q^{2-a_{i_1}-b_{i_2}}r^{a_{i_1}+b_{i_2}+s}}
\end{equation}
where $i_1 \in \{1, \ldots, k\}$, $i_2 \in \{1, \ldots, \ell\}$, $c_q(r)$ is the Ramanujan sum, and $G_{\I}(1-a_{i_1},q)$ and $G_{\J}(1-b_{i_2},q)$ are multiplicative functions  that arise from the additive divisor conjecture (see \eqref{eqn:g-mult}
and \eqref{eqn:G-mult} below). Indeed, in some approximate way, 
\[
  G_{\I}(1-a_{i_1},q) \approx  \sigma_{\I \backslash \{ a_{i_1} \}} (q)
  \quad
  \text{ and }
  \quad
  G_{\J}(1-b_{i_2},q) \approx  \sigma_{\J \backslash \{ b_{i_2} \}} (q). 
\]
One requires a meromorphic continuation of the Dirichlet series $H_{\I,\J; \{ a_{i_1} \} , \{ b_{i_2} \}}(s)$  to the region $\Re(s) \ge  -1$. Furthermore, numerous facts about the gamma function are used; including the beta function identity and various versions of Stirling's formula. At the end, the off-diagonal contribution can be expressed as a sum of contour integrals of the functions $H_{\I,\J; \{ a_{i_1} \} , \{ b_{i_2} \}}(s)$. From these expressions, the integrals corresponding to the diagonal and off-diagonal contributions can be evaluated by a contour shift and the residue theorem.

In order to prove Theorem \ref{thm22}, firstly, we will apply the main theorem of \cite{HN} to our special case. The theorem provides a general asymptotic formula in the form
\[
 \mathscr{D}_{ \sigma_{\I}, \sigma_{\J};\omega}(K) \sim \mathcal{M}_{0, \I, \J; \omega}(K)+ \mathcal{M}_{1, \I, \J; \omega}(K)
 \]  
as $K\to \infty$, where the terms on the right-hand side are as in \eqref{M0swaps} and \eqref{M1swaps}. We will prove in Lemma \ref{thm:holomorphy} that both $\mathcal{M}_{0, \I, \J; \omega}(K)$ and $ \mathcal{M}_{1, \I, \J; \omega}(K)$ are holomorphic as functions of elements of the sets $\I= \{ a_1, \ldots, a_k \}$ and $\J = \{ b_1, \ldots, b_{\ell} \}$. Note that if $k=\ell=2$ and $a_j=b_j=0$ for  $j=1,2$, then $ \mathscr{D}_{ \sigma_{\I}, \sigma_{\J};\omega}(K)$ becomes $  \mathscr{D}_{2,2;\omega}(K) $. Upon explicit computations, each of the main terms $ \mathcal{M}_{0, \I, \J; \omega}(K)$ and $ \mathcal{M}_{1, \I, \J; \omega}(K)$ will be expressed as a sum of polar terms in $a, b, a-b$ or $a+b$. We will then carefully analyze all the terms, and show that these polar terms cancel each other while the remaining terms match the ones in our main theorem. 

This idea of working with the shifted divisor functions  $\sigma_{\I}(n)$ and  $ \sigma_{\J}(n)$ and then setting all variables equal to zero was originated by Ingham \cite{In}.  An advantage of this approach is that when computing the residues one only deals with simple poles. Still, it is quite technical to find a formula for the mean value in terms of the shifts and show that that the polar terms are indeed cancelled out. On the other hand, it is also possible to compute $  \mathscr{D}_{2,2;\omega}(K) $ directly, that is, without using any shifts. In that case, one must deal with poles of higher order, so the residue calculations will be more complicated.  


We now comment on the use of the smooth weight functions $w$ and $\varphi$ in our definition of $\mathscr{D}_{k, \ell;\omega}(K)$. Note that the function $\varphi$ appears in the definitions of $\mathds{A}_{a,\varphi}(s)$ and  $\mathds{B}_{b,\varphi}(s)$. Classical forms of the approximate functional equation do not have smooth weights and they have much weaker error terms. In comparison, weighted approximate equations have much smaller error terms (see \cite[(4.20.1),(4.20.2)]{Ti} and \cite[Theorem 5.3]{IK}). By introducing the function $\varphi$, one is able to make use of the Mellin transform instead of Perron's formula. This has the advantage of providing much better decay rates in the resulting complex integrals. The other weight function $\omega$ can be thought of as a smooth approximation to the indicator function $ \mathds{1}_{[T,2T]}(t)$. The purpose of weighing the mean value with such a function is to improve the estimation of the off-diagonal terms. As in \cite[(4.17) and (4.18)]{HN}, for example, employing the bound in \eqref{whatbound} for $\widehat{\omega}$ allows one to dispense of many error terms.

We remind the reader that mean values of long Dirichlet polynomials are known to be closely related to the moments of the Riemann zeta function (see \cite{CG} and  \cite{Iv2, Iv3}). The $(2k)$-th moment is defined as
 \begin{equation*}
     I_{k}(T) = \int_{0}^{T} |\zeta(\tfrac{1}{2}+it)|^{2k} \, dt.
    \end{equation*} 
For $I_2(T)$,  the fourth moment of the Riemann zeta function, Heath-Brown \cite{HB} was the first to show that it is asymptotic to $T\mathcal{P}_4(\log T)$ for a certain polynomial $\mathcal{P}_4$ of degree four as $T\to \infty$. However, he did not compute all coefficients of this polynomial. Later, Conrey \cite{C} gave several formulae for the coefficients of this polynomial. Finally in \cite{CFKRS2}, numerical values for all coefficients of $\mathcal{P}_4$ were provided by Conrey et al. Now, by the formulae in Conjecture \ref{conjCG}, it is proposed that the first few polynomials for the asymptotics of the moments of Dirichlet polynomials are
\begin{align}
  w_2(x)  =  &\,
 -x^4+8x^3-24x^2+32x-14, \notag \\
 w_3(x)  = & \,  -2x^9+27x^8-324x^7+2268x^6-8694x^5 \notag \\
 &\,\, +19278x^4  -25452x^3  +19764x^2
  -8343x+1479.  \notag
\end{align}
As it turns out, the polynomials $w_3(x)$ and $w_4(x)$ are intimately related to the third and the fourth moments of the Riemann zeta-function, respectively. Indeed, the identities 
\begin{equation*}
  w_3(x) + w_3(3-x) = 42
\quad
\text{
and }
\quad
  2w_4(2) = w_4(2) + w_4(2) = 24024
\end{equation*}
led to Conrey and Gonek's conjectures ~\cite{CG}
\[
  I_3(T) \sim \frac{42 a_3}{9!} T \log^9 T 
  \quad
  \text{and} \quad
  I_4(T) \sim \frac{24024 a_4}{16!} T \log^{16} T. 
\]
Their work also provided a heuristic argument showing that $I_k(T)$ could be expressed as a sum of two mean values of long Dirichlet polynomials with $k$-th divisor functions for $k=3,4$ as in Conjecture \ref{conjCG}.  

Finally, we note that with the same approach as in this article, it is likely to establish an asymptotic formula for $   \mathscr{D}_{k,2;\omega}(K) $ for each other integer $k \geq 3$ for some $K=T^{1+\eta_k}$ where $0<\eta_k<1$, by building on the ideas in \cite{Dr} and \cite{T2, T3}. This is current work in progress. However, this approach would not allow one to estimate $ \mathscr{D}_{k,k;\omega}(K)$ for $K \ge T^2$. In \cite{CK4}, \cite{CK5}, Conrey and Keating introduced a method with new divisor sums to estimate $ \mathscr{D}_{k,k;\omega}(K)$ for such $K$. This created a new branch in this area of research, which is active at the present time.

 \noindent {\bf Conventions and Notation.} 
 
 Given two functions $f(x)$ and $g(x)$, we shall interchangeably use the notations  $f(x)=O(g(x))$, $f(x) \ll g(x)$ and $g(x) \gg f(x)$  to mean that there exists $M >0$ such that $|f(x)| \le M |g(x)|$ for all sufficiently large $x$. The statement $f(x) \asymp g(x)$ means that the estimates $f(x) \ll g(x)$ and $g(x) \ll f(x)$ simultaneously hold.  
  
Per our notation, $\varepsilon$  denotes an arbitrarily small positive constant which may vary from instance to instance. The letter $p$ will always be used to denote a prime number. We also adopt the usual notation that for $s \in \mathbb{C}$, its real part is $\sigma =\Re(s)$. The integral notation
\begin{equation*}
   \label{intc}
  \int_{(c)} f(s)\; ds =: \int_{c-i\infty}^{c + i \infty} f(s) \; ds. 
\end{equation*}
for a complex function $f(s)$ and real number $c $  will be used frequently.
 
Give two sequences $( a(n) ), ( b(n) )$, we define their additive convolution $( (a \star b)(n) )$ by 
\begin{equation*}
  \label{addconv}
  (a \star b)(n) = \sum_{
  \substack{ u,v \ge 0 \\ u+v=n}} a(u) b(v).
\end{equation*} 
This is so that
\begin{equation*}
 \label{multts}
 \Big( \sum_{n=0}^{\infty} a(n) X^n \Big) \Big( \sum_{n=0}^{\infty} b(n) X^n \Big) = \sum_{n=0}^{\infty}  (a \star b)(n) X^n 
\end{equation*}
 for a variable $X$. We will also use the notation  $(-1)^{\bullet}$ to denote the sequence $ ( (-1)^n )_{n=0}^{\infty}$. 
 

 \noindent {\bf Organization. } 
 
The plan of our paper is as follows. In Section \ref{sec:prelim} we define some special functions and fix the notation that will be used throughout the paper. In Section \ref{sec:HN old thm}, we recall the main theorem in \cite{HN}, which provides an asymptotic formula for $\mathscr{D}_{ \sigma_{\I}, \sigma_{\J};\omega}(K) $. We prove that the main terms  $\mathcal{M}_{0, \I, \J; \omega}(K)$ and $\mathcal{M}_{1, \I, \J; \omega}(K)$ in this formula  are holomorphic functions of the elements of $ \, \mathcal I $ and $\mathcal J$. Then in Section \ref{sec:proof of mainthm}, we prove Theorem \ref{thm22} by computing  $\mathcal{M}_{0, \I, \J; \omega}(K)+ \mathcal{M}_{1, \I, \J; \omega}(K)$ explicitly in a special case of $|\mathcal I| = |\mathcal J| = 2$. In Appendix \ref{proof of cor}, we show that  Theorem \ref{thm22} will still hold if the weight function $w$ in the mean value is replaced by $1_{[T, 2T]}$ and thus prove Corollary \ref{cor}. Finally in Appendix \ref{coeff-comp}, we rewrite the expressions for $Q_0(x,y), Q_1(x,y), Q_2(x,y)$, and $Q_3(x,y)$ that appear in Theorem \ref{thm22}  in terms of the $\gamma_{j}$ and $\zeta^{(j)}(2)$ for suitable $j$.


\section{Setting and Preliminaries}\label{sec:prelim}

\subsection{Properties of $\varphi$} \label{ss-phi}

For a fixed number $\mu  \in \big(0,\tfrac{1}{2}\big)$, let $\varphi$ be a smooth, non-negative function defined on $\R_{\ge 0}$ such that
\begin{equation}\label{eq:cond varphi}
 \begin{split}
     & \bullet \varphi(t)=1 \quad  \text{for} \quad 0\leq t \leq 1, \\
    &  \bullet \varphi(t)=0  \quad  \text{for} \quad t\geq 1+\mu,\\
    &    \bullet    \varphi^{(j)}(t) \ll \mu^{-j}  \quad \text{for all} \,\, j \geq 0.
\end{split}
\end{equation}
Its Mellin transform is
    \begin{equation}\label{Phi}
     \Phi(s) =\int_{0}^{\infty} \varphi(t) t^{s-1} dt,
    \end{equation}
which converges absolutely for $\Re(s) >0$. 
The function $\Phi$ has an analytic continuation to  the entire complex plane with the exception of a simple pole at $s=0$ with residue $1$.  


 For $c>0$ and $\Re(s) >c$, we define
\begin{equation}\label{phi2}
  \Phi_2(s) = \frac{1}{2 \pi i} \int_{(c)} \Phi(s_1) \Phi(s-s_1) \, ds_1. 
\end{equation}
Observe that 
\begin{equation}
   \label{phi2mtimt}
     \Phi_2(s) = \int_{0}^{\infty} \varphi(t)^2 t^{s-1} \, dt
    \quad \text{ and } \quad
      \varphi(t)^2 = \frac{1}{2 \pi i} \int_{(c)} \Phi_2(s) t^{-s} \, ds \quad
      \text{ for }\,\, c>0.
\end{equation}
Note that $\Phi_2(s)$ has a simple pole at $s=0$. It also satisfies the bound 
\begin{equation} \label{Phibd}
  |\Phi_2(s)|  
   \ll_{m}  \frac{\mu^{1-m} (1+\mu)^{\sigma+m-1} }{|s(s+1) \cdots (s+m-1)|}
\end{equation}
for $m \ge 1$ and $s \in \mathbb{C} \setminus \{ 0, -1, \ldots, -(m-1) \}$.


\subsection{Taylor Expansions of Some  Functions.} \label{subsec: setup entire fncs}

First, we recall the definitions of the functions $f, h$ and $G$:
\begin{equation}\label{eq:fhG}
\begin{split}
        f(s)  = s \, \zeta(1+s) & = \sum_{j=0}^{\infty} g_j s^j, \\
         h(s)= \frac{1}{\zeta(2+s)} & = \sum_{j=0}^{\infty} \delta_j s^j, \\
         G(s)  =  -2 \int_{0}^{\infty}  \varphi(t) \varphi'(t) t^s \, dt &=  \sum_{j=0}^{\infty} c_j s^j. \\
  \end{split}
\end{equation}
We will provide precise formulae for these coefficients, $ g_j, \delta_j,$ and $c_j$. Recall that 
\begin{equation}\label{zeta laurent}
  \zeta(s) = \frac{1}{s-1} + \sum_{j=0}^{\infty} \frac{(-1)^j \gamma_j}{j!} (s-1)^j, 
\end{equation}
where for $j \ge 0$
\[
  \gamma_j = \lim_{m \to \infty}  \Big( \sum_{k=1}^{m} \frac{(\log k)^j}{k} - \frac{(\log m)^{j+1}}{j+1} \Big).
\]
It follows that the function $f(s)=s\zeta(1+s)$ is entire with the Taylor series expansion
\begin{equation*}
  \label{fts}
  f(s) = \sum_{j=0}^{\infty} g_j s^j,
\end{equation*}
where 
\begin{equation}
  \label{gj}
 g_0 = 1  \quad \text{ and } \quad
 g_{j} = \frac{(-1)^{j-1} \gamma_{j-1}}{(j-1)!}  \quad \text{ for } \, \, j \ge 1. 
\end{equation}
Observe that $g_1 =\gamma_0=\gamma$ is Euler's constant.

We also have 
\begin{equation}
  \label{zeta2ts}
  \frac{1}{\zeta(2+s)} = \sum_{j=0}^{\infty} \delta_j s^j,
\end{equation}
where 
\begin{align}
&\delta_0=\frac{1}{\zeta(2)}, \quad \delta_1=-\frac{\zeta'(2)}{\zeta(2)^2}, \quad \delta_2=\frac{ 2(\zeta'(2))^2-\zeta(2)\zeta''(2)}{\zeta(2)^3}, \notag \\
& \delta_3= \frac{-6(\zeta'(2))^3-\zeta'''(2)\zeta(2)^2+6\zeta(2)\zeta'(2)\zeta''(2)}{\zeta(2)^4}, \label{deltaj}
\\
&\delta_4= \frac{24(\zeta'(2))^4-\zeta^{(4)}(2)\zeta(2)^3+6\zeta(2)^2(\zeta''(2))^2+8\zeta'''(2)\zeta(2)^2\zeta'(2)-36\zeta(2)(\zeta'(2))^2\zeta''(2)}{\zeta(2)^5}, \notag
\end{align}
 and in general $\delta_j$ lies in the field generated by the $\zeta^{(j)}(2)$ with $j \in \N$. 

Now, since 
 \begin{equation*}
  G(s)  = -2 \int_{0}^{\infty}  \varphi(t) \varphi'(t) 
    t^s \, dt=s\Phi_2(s),
\end{equation*}
for $\Phi_2(s)$ as defined in \eqref{phi2mtimt}, it follows that $G(s)$ is an entire function and the coefficients of its Taylor series are given by
\begin{equation*}
  \label{cj}
   c_j = \frac{(-2)}{j!} \int_{0}^{\infty} \varphi(t) \varphi'(t) (\log t)^j \, dt \quad \text{for} \quad j=0, 1, \dots. 
\end{equation*}
In particular, we have   $c_0 = -2 \int_{0}^{\infty} \varphi(t) \varphi'(t) \, dt = \varphi(1)^2=1$. 
Furthermore,  we find that
\begin{equation*}
  \label{cjbound}
  |c_j| \ll \mu^{-1} \int_{1}^{1+\mu} \log^j (t) \, dt
  \ll \mu^{-1} \cdot \mu \cdot \max_{t \in [1,1+\mu]} \log^j(t)
  = \log^j(1+\mu) \ll \mu^j
\end{equation*}
as long as $\mu \in [0,1]$. 

We also introduce the following entire functions and their Taylor expansions, which will appear in our calculations in Section \ref{sec:proof of mainthm}.
\begin{equation}\label{eq:notation}
\begin{split}
        F(s)  = sf'(s) - f(s) = & s^2 \zeta'(1+s)  =\sum_{j=0}^{\infty} g_j' s^j ,\quad \text{ where } g_j' = (j-1) g_j  , \\
    H(s)=- \frac{\zeta'}{\zeta^2}(2+s) & = \sum_{j=0}^{\infty} \delta_j' s^j ,\quad \text{ where } \delta_j'= (j+1) \delta_{j+1}, \\
     L=\log\frac{t}{2\pi},\qquad  &Y= \log K,\qquad X=Y-L, \\
   E_1(s)= & e^{sL}  = \sum_{j=0}^{\infty} \alpha_j s^j, \\
   E_2(s)= &  e^{sY}  = \sum_{j=0}^{\infty} \beta_j s^j. \\
  \end{split}
\end{equation}



\section{A Mean Value Theorem Under an Additive Divisor Conjecture}\label{sec:HN old thm}

We now recall the main theorem in \cite{HN}, where the second and the third author established an asymptotic formula for mean values of long Dirichlet polynomials with higher order shifted divisor functions, assuming a smoothed additive divisor conjecture for higher order shifted divisor functions. Before we state this result, we shall introduce some necessary notation and recall the statement of the additive divisor conjecture.

We set  \begin{equation*} \label{KL}
     \K =  \{ 1, \ldots, k \} \quad \text{ and } \quad
     \Ll =  \{ 1, \ldots, \ell \}.
    \end{equation*} 
Throughout this section, $\I$ and $\J$ are multisets of complex numbers indexed by $\K$ and $\Ll$ respectively and are given by 
    \begin{equation*}\label{IJ}
      \I = \{ a_1, \ldots, a_k \}  \quad \text{ and } \quad
      \J = \{b_1, \ldots, b_{\ell} \}, 
    \end{equation*}
    such that
    \begin{equation}\label{sizerestriction}
     |a_i|, |b_j|   \ll \frac{1}{\log T} \quad \text{ for } \,\,  i \in \K \text{ and } j \in \Ll
    \end{equation}
and 
\begin{equation}
  \label{sizerestriction-} 
  |a_{i_1} - a_{i_2}| \gg \frac{1}{\log T} 
 \,\,\, \text{ and } \,\,
  |b_{j_1}-b_{j_2}| \gg \frac{1}{\log T}
 \,\,\, \text{ for } \,\,  i_1 \ne i_2 \in \K \,\, \text{ and } \,\,
  j_1 \ne j_2 \in \Ll,
\end{equation}
for some parameter $T \ge 2$. 

Also, for a set of shifts  $\I = \{ a_1, \ldots, a_k \}$ as before, we define a shifted divisor function as
    \begin{equation*}
    \sigma_{\I}(n) = \sum_{d_1 \cdots d_k=n} d_{1}^{-a_1} \cdots d_{k}^{-a_k}. 
    \end{equation*}
Observe that if $\I=\{0, \ldots, 0\}$, then $\sigma_{\I}(n)=\tau_{k}(n)$.

\subsection{The Additive Divisor Conjecture}

We define the shifted convolution sum
\begin{equation}
  \label{DFIJr}
D_{F; \I,\J}(r)=\sum_{\substack{m,n\geq1\\m-n=r}}\sigma_{\I}(m)\sigma_{\J}(n) F(m,n).
\end{equation} 
Here we assume that for some $X, Y$ and $P \ge 1$  
\begin{equation}
  \label{Fsupport}
  \text{support}(F) \subset [X,2X] \times [Y,2Y]
\end{equation}
and that
\begin{equation} \label{Fcond} 
 x^{m} y^{n} F^{(m,n)}(x,y) \ll_{m,n}
P^{m+n}. 
\end{equation}

For a finite multiset of complex numbers $A=\{a_{1},\cdots,a_{m}\}$ and $s\in\mathbb{C}$,  we define two multiplicative functions $n\mapsto g_{A}(s,n)$ and $n\mapsto G_{A}(s,n)$ by
\begin{equation}\label{eqn:g-mult}
  g_{A}(s,n) = \prod_{p^{e}||n}\bigg(\sum_{j=0}^{\infty} \frac{\sigma_{A}(p^{j+e})}{p^{js}}\bigg)\bigg/\bigg( \sum_{j=0}^{\infty} \frac{\sigma_{A}(p^j)}{p^{js}}\bigg)
\end{equation}
and
\begin{equation}\label{eqn:G-mult}
G_{A}(s,n) = \sum_{d \mid n} \frac{\mu(d) d^s}{\phi(d)}
  \sum_{e \mid d} \frac{\mu(e)}{e^s} g_{A} \Big(s, \frac{ne}{d} \Big).
\end{equation}
Notice that for $n\in\mathbb{N}$, we have 
$$
\sum_{n=1}^{\infty}\frac{\sigma_{A}(jn)}{j^{s}}=g_{A}(s,n)\prod_{a\in A}\zeta(s+a).
$$

We are now prepared to state a conjectural asymptotic formula for $D_{F;\I,\J}(r)$.

\begin{conjecture} \label{conj:shifted-sum}\text{($k$-$\ell$ Additive divisor conjecture)} There exists a triple $(\vartheta_{k,\ell}, C_{k,\ell},\beta_{k,\ell}) \in [\frac{1}{2}, 1) \times [0,\infty) \times (0,1]$ for which the following holds (henceforth to be referred to as $\mathcal{AD}_{k,\ell}(\vartheta_{k,\ell},C_{k,\ell},\beta_{k,\ell})$ conjecture). 

Let $\varepsilon$  be  a positive absolute constant, $P > 1$, and $X,Y > \frac{1}{2}$ satisfy $Y \asymp X$.  Let $F$ be a smooth function satisfying the conditions \eqref{Fsupport} and \eqref{Fcond}, and suppose that $\I = \{ a_1, a_2, \ldots, a_k \}$ and $\J = \{ b_1, \ldots, b_{\ell} \}$ are sets of distinct complex numbers such that $|a_i|, |b_j| \ll \frac{1}{\log X}$ for all
$ i \in \{ 1, \ldots, k \}$ and  $j \in \{ 1, \ldots, \ell \}$. Then for $   D_{F;\I,\J}(r)$ as defined in \eqref{DFIJr}, in the cases where $X$ is sufficiently large (in absolute terms), one has
  \begin{align*}
      D_{F;\I,\J}(r)  =&\,   \sum_{i_1=1}^{k} \sum_{i_2=1}^{\ell} 
    \prod_{j_1 \ne i_1} \zeta(1-a_{i_1}+a_{j_1})   \prod_{j_2 \ne i_2} \zeta(1-b_{i_2}+b_{j_2})   \\
   \times & \sum_{q=1}^{\infty} \frac{c_{q}(r)G_{\I}(1-a_{i_1},q)G_{\J}(1-b_{i_2},q)  }{q^{2-a_{i_1}-b_{i_2}}}
  \int_{\max(0,r)}^{\infty}  f(x,x-r) x^{-a_{i_1}}(x-r)^{-b_{i_2}} dx 
  \\
&+ O \left( P^{C_{k,\ell}} X^{\vartheta_{k,\ell}+\varepsilon} \right)
 \end{align*}
 uniformly for $1 \le |r| \ll X^{\beta_{k,\ell}}$.
 \end{conjecture}
 
 We say that $\mathcal{AD}_{k,\ell}(\vartheta_{k,\ell},C_{k,\ell},\beta_{k,\ell})$ holds if the $k$-$\ell$ additive divisor conjecture holds for a triple $(\vartheta_{k,\ell}, C_{k,\ell},\beta_{k,\ell}) \in \big[\tfrac{1}{2}, 1\big) \times [0,\infty) \times (0,1]$. It is important to note that in the case $|\I|=|\J| =2$, Hughes and Young \cite[p.~218]{HY} proved that $\mathcal{AD}_{k,\ell}(\vartheta_{k,\ell},C_{k,\ell},\beta_{k,\ell})$ holds for $\vartheta_{2,2} = \frac{3}{4}$, $C_{2,2} = \frac{5}{4}$ and $\beta_{2,2}=1$ by using  Duke, Friedlander and Iwaniec's $\delta$-method \cite{DFI}. 




\subsection{Mean Values of Long Dirichlet Polynomials with Shifted Divisor Functions as Coefficients} We now consider the mean value of long Dirichlet polynomials associated with the shifted divisor functions $\sigma_{\I}$ and $\sigma_{\J}$ as defined in \eqref{DabwK}. For simplicity, we set
    \[
     \mathscr{D}_{\I,\J;\omega}(K) =
     \mathscr{D}_{\sigma_{\I},\sigma_{\J};\omega}(K).
    \]
   \begin{definition*} 
Let $\I,\J$ be finite multisets of complex numbers. We define $\mathcal{B}(\I,\J)$ as the series 
\begin{align}\label{BIJ}
\mathcal{B}(\I,\J)&=\sum_{n=1}^{\infty}\frac{\sigma_{\I}(n)\sigma_{\J}(n)}{n},
\end{align} 
provided that the series converges (for example, when $\Re(a),\Re(b)>0$ for all $a\in\I$ and $b\in \J$), and by analytic continuation elsewhere.
\end{definition*}
Observe that when the series \eqref{BIJ} converges, we use the multiplicativity of $\sigma_{\I}\sigma_{\J}$ to write
\begin{align*}
  \mathcal{B}(\I,\J)&=\prod_{p}\sum_{u=0}^{\infty}\frac{\sigma_{\I}(p^u)\sigma_{\J}(p^u)}{p^{u}}.
  \end{align*}
Upon factoring out $\prod_{p}  \prod_{\substack{ i \in \K \\ j \in \Ll}  }  (1 - p^{-1-a_i -b_j})^{-1}$  from the right-hand side of this, we obtain 
\begin{equation*}
\begin{split}
  \mathcal{B}(\I,\J)&=\prod_p\prod_{\substack{i \in \K \\ j \in \Ll}}(1 - p^{-1-a_i -b_j})^{-1}\prod_{p}\prod_{\substack{i \in \K \\ j \in \Ll}}(1 - p^{-1-a_i -b_j})\sum_{u=0}^{\infty}\frac{\sigma_{\I}(p^u)\sigma_{\J}(p^u)}{p^{u}}.
\end{split}
\end{equation*}

\begin{definition*}\label{ZA}
For a prime $p$ and $s \in \C$, we set $z_p(s) = (1-p^{-s})^{-1}$. From the local factors $z_{p}(s)$, we define 
\begin{align}
   \label{ZIJ}
     \Z(\I,\J) & =  \prod_{p}\prod_{\substack{i\in \K\\j\in\Ll}} z_{p}(1+a_i+b_j), \\
     \label{AIJ}
      \mathcal{A}(\I,\J) & =  \prod_{p} \prod_{\substack{i\in \K\\j\in\Ll}}  z_{p}^{-1}(1+a_i+b_j)\sum_{u=0}^{\infty}\frac{\sigma_{\I}(p^u)\sigma_{\J}(p^u)}{p^{u}}. 
 \end{align}
\end{definition*}
Observe that we have 
\begin{equation*}
  \label{ZIJid}
   \mathcal{Z}(\I,\J) = \prod_{i \in \K, j \in \Ll} \zeta(1+a_i+b_j)
\end{equation*}
and 
\begin{equation}
  \label{BIJid}
   \mathcal{B}(\I,\J)  =   \mathcal{A}(\I,\J) \mathcal{Z}(\I,\J).
\end{equation}
Next, we require some notation regarding set operations. Given a multiset $U= \{\alpha_1, \ldots, \alpha_n \}$ and $\xi \in \mathbb{C}$, we 
define 
$ U+ \xi = \{\alpha_1+\xi, \ldots, \alpha_n+\xi \}$. We also set $-U=\{-\alpha_1, \ldots, -\alpha_n \}$. With this notation, observe that the identity
\begin{equation}
    \label{sigmaidentity}
    \sigma_{U+\xi}(n) = n^{-\xi} \sigma_{U}(n)
\end{equation}
holds.  

We are now ready to state \cite[Theorem~1.1]{HN}. 

 \begin{theorem}\label{HN-main}
    Let $|\I|=k$ and $|\J|=\ell $ with $ k,\ell \ge 2$, and suppose that both $\I$ and $\J$ satisfy \eqref{sizerestriction} and \eqref{sizerestriction-}.  Assume that 
$\mathcal{AD}_{k,\ell}(\vartheta_{k,\ell},C_{k,\ell},\beta_{k,\ell})$ holds for some triple $(\vartheta_{k,\ell}, C_{k,\ell},\beta_{k,\ell}) \in [\frac{1}{2}, 1) \times [0,\infty) \times (0,1]$. Let $K=T^{1+\eta}$ with $\eta>0$, and let $\omega$ satisfy \eqref{cond1}, \eqref{cond2}, and \eqref{cond3} with $\nu>\tfrac{(1-\beta_{k,\ell})(1+\eta)}{1-\epsilon}$ and $0<\epsilon<1$. Then we have  
\begin{equation}\label{eqn:HN-main}
    \mathscr{D}_{\I,\J;\omega}(K) =   \mathcal{M}_{0, \I, \J; \omega}(K)  +   \mathcal{M}_{1, \I, \J; \omega}(K) 
    +   O \bigg( K^{\vartheta_{k,\ell}+\e} \Big( \frac{T}{T_0} \Big)^{1+C_{k,\ell}} \bigg),
\end{equation}
    where 
        \begin{equation} \label{M0swaps}
          \mathcal{M}_{0, \I, \J; \omega}(K) 
          = \frac{\widehat{\omega}(0)}{2\pi i}\int_{(c)}K^{s}\Phi_{2}(s) \mathcal{B}\left(\I +s,\J\right)\;ds
         \end{equation}
    and
         \begin{equation}
         \begin{split}
          \label{M1swaps}
         \mathcal{M}_{1,\I, \J; \omega}&(K)  \\
          =&\, \int_{0}^{\infty} \omega(t) \sum_{i \in \K,j \in \Ll  }
           \Big( \frac{t}{2 \pi} \Big)^{-a_i -b_j } \mathcal{Z}\big(\I \setminus \{a_i\},\{-a_i\}\big)\mathcal{Z}\big(\{-b_j\},\J \setminus \{ b_j \}\big)
         \\&\times  \frac{1}{2 \pi i}  \int_{(c)} 
         \Phi_2(s)\left(\frac{2\pi K}{t}\right)^{s}\mathcal{Z}\big((\I \setminus \{a_i\})+s, \J \setminus \{b_j\} \big)\zeta(1-a_i-b_j-s)\\
         &\times\mathcal{A} \left( (\I \setminus \{a_i\}) \cup \{ -b_j-s \},   ((\J \setminus \{b_j\})+s)\cup\{-a_i\} \right)\;ds\;dt.
       \end{split}
       \end{equation}
Here $c>0$ is fixed and $\Phi_{2}(s)$ is as defined in \eqref{phi2}. 
\end{theorem}


\subsection{Holomorphy of $\mathcal{M}_{0,\I, \J; \omega}(K)$ and $\mathcal{M}_{1,\I, \J; \omega}(K)$}\label{sec:holomorphy}

We will now prove that the main term in the asymptotic formula \eqref{eqn:HN-main} is holomorphic as a function of the shifts $a_1, \dots, a_k, b_1, \dots, b_\ell$. As a consequence of this, in another lemma we will prove that Theorem \ref{HN-main} holds without the restrictions in \eqref{sizerestriction-}.   

\begin{lemma}\label{thm:holomorphy}
Under the hypothesis of Theorem \ref{HN-main} and with the same definitions, the terms  $\mathcal{M}_{0, \I, \J; \omega}(K) $ and $ \mathcal{M}_{1, \I, \J; \omega}(K) $, which are written in \eqref{M0swaps} and \eqref{M1swaps}, respectively, are both holomorphic as functions of the variables $a_1, \dots, a_k, b_1, \dots, b_\ell$. 
\end{lemma}

\begin{proof}
We follow the argument that was employed in \cite[Section 11]{BTB}. Recall that $a_i, b_j \ll \frac{1}{\log T}$ for all $i \in \mathcal{K}$ and $j \in \mathcal{L}$. 

First, we consider $\mathcal{M}_{0,\I, \J; \omega}(K)$. We repeat \eqref{M0swaps} here.
    \begin{equation*}
          \mathcal{M}_{0, \I, \J; \omega}(K) 
          = \frac{\widehat{\omega}(0)}{2\pi i}\int_{(c)}K^{s}\Phi_{2}(s)\mathcal{B}\left(\I+s,\J\right)\;ds,
    \end{equation*}
where, as in \eqref{BIJid} , we have
    \[
     \mathcal{B}(\I+s,\J)  =   \mathcal{A}(\I+s,\J) \mathcal{Z}(\I+s,\J)
    \]
with
    \[
    \mathcal{A}(\I+s,\J)
    =\prod_{p}\prod_{x \in \I+s, y \in \J} \bigg(1-\frac{1}{p^{1+x+y}}\bigg) \sum_{u=0}^\infty \frac{\sigma_{\I+s}(p^u) \sigma_{\J}(p^u)}{p^u}
    \]
and
    \[
    \mathcal{Z}(\I+s,\J)= \prod_{x \in \I+s, y \in \J} \zeta(1+x+y). 
    \]
From \eqref{BIJ} and \eqref{sigmaidentity}, we see that $ \mathcal{B}(\I+s,\J)$ is holomorphic as a function of the variables $a_i$ and $b_j$ due to the restriction on the size of $a_i$ and $b_j$. Thus $  \mathcal{M}_{0,\I, \J; \omega}(K)$ is also holomorphic in the $a_i$ and the $b_j$. 

We proceed with $\mathcal{M}_{1,\I, \J; \omega}(K)$, which, by \eqref{M1swaps}, is given as
         \begin{align*}
         \mathcal{M}_{1,\I, \J; \omega}(K)  
         & = \int_{0}^{\infty} \omega(t) \sum_{i_0 \in \K, j_0 \in \Ll  }
           \Big( \frac{t}{2 \pi} \Big)^{-a_{i_0} -b_{j_0} } \mathcal{Z}\big(\I \setminus \{a_{i_0}\},\{-a_{i_0}\}\big)
           \mathcal{Z}\big(\{-b_{j_0}\},\J \setminus \{ b_{j_0}\}\big)
         \\
          \times  &\frac{1}{2 \pi i}  \int_{(c)} 
         \Phi_2(s)\left(\frac{2\pi K}{t}\right)^{s}\mathcal{Z}\big((\I \setminus \{a_{i_0}\})+s, \J \setminus \{b_{j_0}\} \big)\zeta(1-a_{i_0}-b_{j_0}-s)\\
         &\hspace{5em}\times \mathcal{A} \big( (\I \setminus \{a_{i_0}\}) \cup \{ -b_{j_0}-s \},   ((\J \setminus \{b_{j_0}\})+s)\cup\{-a_{i_0}\} \big) \;ds\;dt.
       \end{align*}
For now, we assume that both sets $\mathcal{I}$ and $\mathcal{J}$ have distinct elements and that their intersection is empty. We expand each $ \mathcal{Z}$-term and $ \mathcal{A}$-term in the above. By definition, 
\begin{equation}\label{eq:product of zetas 1}
 \mathcal{Z}\big(\I \setminus \{a_{i_0}\},\{-a_{i_0}\}\big)
 = \prod_{\substack{x\in \mathcal{I}\setminus \{a_{i_0}\},\\ y \in \{-a_{i_0}\} }}
 \zeta(1+x+y)
 = \prod_{i \neq i_0} \zeta(1+a_i-a_{i_0})
       \end{equation}
and 
\begin{equation}\label{eq:product of zetas 2}
 \mathcal{Z}\big(\{-b_{j_0}\},\J \setminus \{ b_{j_0}\}\big)
 =  \prod_{\substack{x\in \mathcal{J}\setminus \{b_{j_0}\},\\ y \in \{-b_{j_0}\} }}
 \zeta(1+x+y)
 = \prod_{j \neq j_0} \zeta(1+b_j-b_{j_0}) .
       \end{equation}
Also, by an argument of inclusion-exclusion we have
     \begin{equation}\label{eq:product of zetas 3}
     \begin{split}
      \mathcal{Z}\big(\{-b_{j_0}\},\J \setminus \{ b_{j_0}\}\big)
      = 
    & \prod_{\substack{i\in \mathcal{K}, j\in \mathcal{L},\\ i\neq i_0, j\neq j_0}}  \zeta(1+a_i+s+b_j)
     \\
    =& \frac{ \prod_{i\in \mathcal{K}, j\in \mathcal{L}} \zeta(1+a_i+s+b_j)}
   { \prod_{i\in \mathcal{K}}  \zeta(1+a_{i}+s+b_{j_0})   
    \prod_{j\in \mathcal{L}} \zeta(1+a_{i_0}+s+b_j)   }
    \zeta(1+a_{i_0}+s+b_{j_0}). 
       \end{split}  
       \end{equation}
For the $\mathcal A$-term as defined via \eqref{AIJ}, we note the following for its Euler product part.  
     \begin{align*}
       \prod_{\substack{x \in (\I \setminus \{a_{i_0}\}) \cup \{ -b_{j_0}-s \}, \\ y \in 
       ((\J \setminus \{b_{j_0}\})+s)\cup\{-a_{i_0}\} }} 
       \bigg(1-\frac{1}{p^{1+x+y}}\bigg)  
       =
         \prod_{\substack{x \in  \I \setminus \{a_{i_0}\}  , \\ y \in (\J \setminus \{b_{j_0}\})+s }}  
         \bigg(1-\frac{1}{p^{1+x+y}}\bigg)   
      \prod_{\substack{x \in  \I \setminus \{a_{i_0}\}, \\ y \in \{-a_{i_0}\} }}
       \bigg(1-\frac{1}{p^{1+x+y}}\bigg)  & \\
      \times \prod_{\substack{x \in  \{ -b_{j_0}-s \}, \\ y \in   (\J \setminus \{b_{j_0}\})+s }}
       \bigg(1-\frac{1}{p^{1+x+y}}\bigg)  
       \prod_{\substack{x \in  \{ -b_{j_0}-s \}, \\ y \in  \{-a_{i_0}\} }}
       \bigg(1-\frac{1}{p^{1+x+y}}\bigg).
       \end{align*}
Again by inclusion-exclusion, this can also be written as
     \begin{equation}\label{eq:1 - 1 over p}
     \begin{split}
       \prod_{\substack{x \in (\I \setminus \{a_{i_0}\}) \cup \{ -b_{j_0}-s \}, \\ y \in 
       ((\J \setminus \{b_{j_0}\})+s)\cup\{-a_{i_0}\} }} 
       \bigg(1-\frac{1}{p^{1+x+y}}\bigg)  
       = \prod_{i\in \mathcal{K}, j\in \mathcal{L}} \bigg(1-\frac{1}{p^{1+a_i+b_j+s}}\bigg)           
    \prod_{i\in \mathcal{K}}    \bigg(1-\frac{1}{p^{1+a_i+b_{j_0}+s}}\bigg)^{-1}   \\
    \times \prod_{j\in \mathcal{L}} \bigg(1-\frac{1}{p^{1+a_{i_0}+b_j+s}}\bigg)^{-1} 
     \times \bigg(1-\frac{1}{p}\bigg)^{-2}
      \prod_{i\in \mathcal{K}}   \bigg(1-\frac{1}{p^{1+a_i-a_{i_0}}}\bigg) &  \\
       \times  \prod_{j\in \mathcal{L}}   \bigg(1-\frac{1}{p^{1+b_j-b_{j_0}}}\bigg) 
      \bigg(1-\frac{1}{p^{1+a_{i_0}+b_{j_0}+s}}\bigg) 
     \bigg(1-\frac{1}{p^{1-a_{i_0}-b_{j_0}-s}}\bigg).  &
       \end{split}
       \end{equation}  
In view of these expressions, it will be useful to define for each prime $p$ 
    \begin{equation*}
    \begin{split} 
     \mathcal{P}(p) = &\, \, \mathcal{P}(z_1, z_2, s; p) \\
     := &   \prod_{i \in \K, j\in \mathcal{L}} \bigg(1-\frac{1}{p^{1+a_i+b_j+s}}\bigg)  
           \prod_{i \in \K} \bigg(1-\frac{1}{p^{1+a_i-z_2+s}}\bigg)^{-1}
      \prod_{j \in \mathcal{L}} \bigg(1-\frac{1}{p^{1+b_j-z_1+s}}\bigg)^{-1} \\
    \times   \bigg(1-\frac1p\bigg)^{-2}  &  \prod_{i \in \K} \bigg(1-\frac{1}{p^{1+a_i+z_1}}\bigg) 
         \prod_{j \in \mathcal{L}} \bigg(1-\frac{1}{p^{1+b_j+z_2}}\bigg)  
    \bigg(1-\frac{1}{p^{1-z_1-z_2+s}}\bigg)  \bigg(1-\frac{1}{p^{1+z_1+z_2-s}}\bigg)  .
    \end{split} 
    \end{equation*} 
By using Cauchy's theorem, we can now write $\mathcal{M}_{1,\I, \J; \omega}(K)$ as a sum of residues and thus as an integral. We have by \eqref{eq:product of zetas 1}, \eqref{eq:product of zetas 2} and \eqref{eq:product of zetas 3}
     \begin{equation}\label{eq:M1 as  residue}
     \begin{split}
     \mathcal{M}_{1,\I, \J; \omega}(K)
     = \int_0^\infty \omega(t)  \sum_{i \in \K, j\in \mathcal{L}} \bigg(\frac{t}{2\pi}\bigg)^{z_1+z_2} 
    \frac{1}{2\pi i}  \int_{(c)} 
       \Phi_2(s) \Big(\frac{2\pi K}{t}\Big)^{s}  & \\
    \frac{1}{(2\pi i)^2}  \int_{|z_1|= \frac{c}{4}}\int_{|z_2|= \frac{c}{4}} 
    \prod_{i\in \mathcal{K}}  \zeta(1+z_1+a_{i}) 
     \prod_{j\in \mathcal{L}}  \zeta(1+z_2+b_{j})   & \\
    \times \, \zeta(1+z_1+z_2-s)  \zeta(1-z_1-z_2+s)  & \\
     \times \,  \frac{\prod_{i\in \mathcal{K}, j\in \mathcal{L}}  \zeta(1+a_i+b_j+s)}
    {\prod_{i\in \mathcal{K}}  \zeta(1+a_{i}-z_2+s)   
    \prod_{j\in \mathcal{L}} \zeta(1-z_1+b_{j}+s) } & \\
    \times \,  \prod_{p}  \mathcal{P}(p)       
    \sum_{u=0}^\infty \frac{\sigma_{(\I \setminus \{-z_1\}) \cup \{z_2-s \}}(p^u) \sigma_{ ((\J \setminus \{-z_2\})+s)\cup\{z_1\}}(p^u)}{p^u} dz_1 dz_2 ds dt. &
    \end{split}
    \end{equation}
This is because the pairs $z_1=-a_{i}$ and $z_2=-b_{j}$ for $i \in \mathcal{K}, j \in \mathcal L$ are the only poles of the above integrand, all of which are simple. 

Moreover, the integrand is holomorphic as a function of the $a_i, b_j$ whenever they are distinct as per our assumption. This is clear to see for the part of the integrand that involves $\zeta$-values. It thus remains to show that the Euler product in the above converges absolutely. For this, note that by \eqref{eq:1 - 1 over p} we have
    \begin{equation}\label{eq: P p expanded}
    \begin{split}                
&\mathcal{P}(p)       
    \sum_{u=0}^\infty \frac{\sigma_{(\I \setminus \{-z_1\}) \cup \{ z_2-s \}}(p^u) \sigma_{ ((\J \setminus \{-z_2\})+s)\cup\{z_1\}}(p^u)}{p^u} \\
=
& \prod_{\substack{x \in (\I \setminus \{-z_1\}) \cup \{z_2-s \}, \\ y \in 
     ((\J \setminus \{-z_2\})+s)\cup\{z_1\} }} \bigg(1-\frac{1}{p^{1+x+y}}\bigg)       
      \sum_{u=0}^\infty 
   \frac{\sigma_{(\I \setminus \{-z_1\}) \cup \{z_2-s \}}(p^u) \sigma_{ ((\J \setminus \{-z_2\})+s)\cup\{z_1\}}(p^u)}{p^u}  \\
= & \bigg(1+ \frac{  \sigma_{(\I \setminus \{-z_1\}) \cup \{z_2-s \}}(p) \sigma_{ ((\J \setminus \{-z_2\})+s)\cup\{z_1\}}(p)}{p} \bigg)
     \prod_{\substack{x \in (\I \setminus \{-z_1\}) \cup \{z_2-s \}, \\ y \in 
     ((\J \setminus \{-z_2\})+s)\cup\{z_1\} }}  \bigg(1-\frac{1}{p^{1+x+y}}\bigg)      \\
 &+O_{\varepsilon}\bigg(\frac{1}{p^{2-8c+\varepsilon}}\bigg).        
    \end{split}      
    \end{equation}          
In the last step, we used the estimate
\begin{equation}\label{eq:sigma sum}
      \sum_{u=2}^\infty 
   \frac{\sigma_{(\I \setminus \{-z_1\}) \cup \{z_2-s \}}(p^u) \sigma_{ ((\J \setminus \{-z_2\})+s)\cup\{z_1\}}(p^u)}{p^u} 
     \ll_{\varepsilon} \frac{1}{p^{2-8c+\varepsilon}} \quad \text{ for suitable } \, \varepsilon>0.
\end{equation}
This estimate follows from the fact that for any $\varepsilon>0$
\[
\sigma_{(\I \setminus \{-z_1\}) \cup \{z_2-s \}}(p^u) 
\ll_{\varepsilon} p^{u\big(-\min_{v \in (\I \setminus \{-z_1\}) \cup \{z_2-s \} } \{\Re v \} +\varepsilon\big)} 
\ll p^{u(2c+\varepsilon)}
\]
since
\[
\Re(v) \gg -c -\frac{c}{4} \geq -2c \quad \text{for} \quad v \in (\I \setminus \{-z_1\}) \cup \{z_2-s \},
\]
and the similar estimate
\[
 \sigma_{ ((\J \setminus \{-z_2\})+s)\cup\{z_1\}}(p^u) \ll_{\varepsilon} p^{u(2c+\varepsilon)}.
\]
On the other hand, as in \cite[(9.1.4)]{BTB}, we observe that for suitable $\varepsilon >0$ 
    \begin{align*}   
     & \prod_{\substack{x \in (\I \setminus \{-z_1\}) \cup \{z_2-s \}, \\ y \in 
     ((\J \setminus \{-z_2\})+s)\cup\{z_1\} }}  \bigg(1-\frac{1}{p^{1+x+y}}\bigg)    
     \\
     =&\,  1-   \frac{  \sigma_{(\I \setminus \{-z_1\}) \cup \{z_2-s \}}(p) \sigma_{ ((\J \setminus \{-z_2\})+s)\cup\{z_1\}}(p)}{p}
      + O\bigg(\frac{1}{p^{1+\varepsilon}}\bigg). 
    \end{align*}     
By combining this with \eqref{eq:sigma sum}, we obtain
     \begin{align*}  
&  \bigg(1+    \frac{  \sigma_{(\I \setminus \{-z_1\}) \cup \{z_2-s \}}(p) \sigma_{ ((\J \setminus \{-z_2\})+s)\cup\{z_1\}}(p)}{p} \bigg)
     \prod_{\substack{x \in (\I \setminus \{-z_1\}) \cup \{z_2-s \}, \\ y \in 
     ((\J \setminus \{-z_2\})+s)\cup\{z_1\} }} \bigg(1-\frac{1}{p^{1+x+y}}\bigg)    \\
    = & \,\, 1+ O\bigg(\frac{1}{p^{1+\varepsilon}}\bigg)+O_{\varepsilon}\bigg(\frac{1}{p^{2-8c+\varepsilon}}\bigg). 
          \end{align*}   
for $\varepsilon>0$ sufficiently small. Finally by \eqref{eq: P p expanded} and by choosing $c>0$ suitably, we deduce that the Euler product in \eqref{eq:M1 as  residue} converges absolutely, hence it is holomorphic in the $a_i$ and $b_j$. Therefore, we have shown that if both $\mathcal I$ and $ \mathcal J$ have no repeated elements and that they don't have any elements in common, then the right-hand side of \eqref{eq:M1 as  residue} is a holomorphic function of the $a_i$ and $b_j$. By analytic continuation, the same expression, and thus $ \mathcal{M}_{1,\I, \J; \omega}(K)$, is a holomorphic function of the shifts $a_1, \dots, a_k, b_1, \dots, b_\ell$ that satisfy the condition $a_i, b_j \ll \frac{1}{\log T}$ for all $i, j$. 
\end{proof}

\begin{lemma}\label{thm:remove-size-cond}
 Theorem \ref{HN-main} holds without assuming the size restriction in \eqref{sizerestriction-}.\end{lemma}

 \begin{proof}
 We follow the argument that was employed in \cite[Section~5]{Ng}.  We set ${\bf{a}}=(a_1,a_2,\dots,a_k)$ and ${\bf{b}}=(b_1,b_2,\dots,b_{\ell})$. We also let $L({\mathbf a},{\mathbf b})=  \mathscr{D}_{\I,\J;\omega}(K)$ and $R({\bf a},{\bf b})=\mathcal{M}_{0, \I, \J; \omega}(K)  +   \mathcal{M}_{1, \I, \J; \omega}(K)$ for convenience. By Theorem \ref{HN-main}, we know that 
 \begin{equation}\label{eqn:error-bd}L({\mathbf a},{\mathbf b})-R({\mathbf a},{\mathbf b})=O \bigg( K^{\vartheta_{k,\ell}+\e} \Big( \frac{T}{T_0} \Big)^{1+C_{k,\ell}} \bigg),\end{equation}
 provided that conditions \eqref{sizerestriction} and \eqref{sizerestriction-} are satisfied.
 By Definition \ref{DabwK}  and Lemma \ref{thm:holomorphy}, we also know that  $L({\bf a},{\bf b})-R({\bf a},{\bf b})$ is holomorphic as a function of the variables $a_1,\dots,a_k,b_1,\dots,b_{\ell}$.
 
 Suppose that $a_1,\dots,a_k,b_1,\dots,b_{\ell}$ are complex numbers satisfying $|a_{j}|,|b_j|\leq \frac{C_0}{\log T}$ for some positive constant $C_0$. Consider the polydisc $D\subset\mathbb{C}^{k+\ell}$ given by 
 \[
 D=\prod_{j=1}^kD_j\prod_{j=1}^{\ell}\tilde{D}_j,
 \]
 where 
 \[
 D_j=\{z\in\mathbb{C}:|z-a_j|\leq r_j\},\quad 
 \tilde{D}_j=\{z\in\mathbb{C}:|z-b_j|\leq r_j\}
\quad
\text{and}
\quad
 r_j=\frac{2^{j+1}C_0}{\log T}.
 \]

 Let $\partial{D}_j$ and $\partial{\tilde{D}_j}$ be the boundaries of the discs $D_j$ and $\tilde{D}_j$ respectively. By Cauchy's integral formula, we have
 \begin{equation}\label{eqn:cauchy}
 L({\mathbf a},{\bf b})-R({\bf a},{\mathbf b})=\frac{1}{(2\pi i)^{k+\ell}}\int_{\partial{D}_1}\cdots\int_{\partial{D}_{k}}\int_{\partial{\tilde{D}}_1}\cdots\int_{\partial{\tilde{D}}_{\ell}}\frac{L({\mathbf{z}},{\mathbf{w}})-R({\mathbf{z}},{\mathbf{w}})}{({\mathbf z}-{\mathbf a})({\mathbf w}-{\mathbf b})}\;{\mathbf{dz}}\; {\mathbf{dw}},
 \end{equation}
 where ${\bf dz}=dz_1\cdots dz_k$, ${\mathbf{dw}}=dw_1\cdots dw_{\ell}$, ${\mathbf z}-{\bf a}=\prod_{j=1}^k(z_j-a_j)$ and ${\mathbf w}-{\mathbf b}=\prod_{j=1}^{\ell}(w_j-b_j)$.

 Observe that for $1\leq j_2<j_1\leq k$ we have
  \[
  |z_{j_1}-z_{j_2}|\geq |z_{j_1}-a_{j_1}|-|z_{j_2}-a_{j_2}|-|a_{j_1}|-|a_{j_2}|\geq \frac{2C_0}{\log T},
 \]
 and 
 \[
 |w_{j_1}-w_{j_2}|\geq |w_{j_1}-b_{j_1}|-|w_{j_2}-b_{j_2}|-|b_{j_1}|-|b_{j_2}|\geq \frac{2C_0}{\log T}.
 \] 
 Hence $z_j$ and $w_j$ satisfy the conditions  \eqref{sizerestriction} and \eqref{sizerestriction-}. In particular, \eqref{eqn:error-bd} holds for $(z_1,\dots, z_k)\in\prod_{j=1}^k\partial{D}_j$ and $(w_1,\dots,w_{\ell})\in\prod_{j=1}^{\ell}\partial{\tilde{D}}_j$. More precisely, we have
\begin{equation*}
L({\mathbf z},{\mathbf w})-R({\mathbf z},{\mathbf w})=O \bigg( K^{\vartheta_{k,\ell}+\e} \Big( \frac{T}{T_0} \Big)^{1+C_{k,\ell}} \bigg).
\end{equation*}
By using this bound in \eqref{eqn:cauchy}, we obtain 
\begin{align*}
L({\mathbf a},{\bf b})-R({\bf a},{\mathbf b})
&\ll K^{\vartheta_{k,\ell}+\e} \Big( \frac{T}{T_0} \Big)^{1+C_{k,\ell}}
 \prod_{j=1}^{k}\frac{\text{length}\partial{D_j}}{r_j}\prod_{j=1}^{\ell}\frac{\text{length}(\partial{\tilde{D}_j})}{r_j} \\
&\ll K^{\vartheta_{k,\ell}+\e} \Big( \frac{T}{T_0} \Big)^{1+C_{k,\ell}},
\end{align*}
 as desired.

 \end{proof}


\section{Proof of Theorem \ref{thm22} }\label{sec:proof of mainthm}

As a first step in proving Theorem \ref{thm22}, we shall apply Theorem \ref{HN-main} with $\I=\{a,0\}$ and $\J=\{b,0\}$. In the case $|\I|=|\J| =2$, we know that $\mathcal{AD}_{k,\ell}(\vartheta_{k,\ell},C_{k,\ell},\beta_{k,\ell})$ holds with $\vartheta_{2,2} = \frac{3}{4}$, $C_{2,2} = \frac{5}{4}$, and $\beta_{2,2}=1$ \cite[p. 218]{HY}. Hence, Theorem \ref{HN-main} holds unconditionally for any $\eta<\frac13$.

In order to compute $ \mathscr{D}_{2,2;\omega}(K) $, we will simplify the expressions for $\mathcal{M}_{0,\I, \J; \omega}(K)$ and  $\mathcal{M}_{1,\I, \J; \omega}(K)$ that were provided by Theorem \ref{HN-main}. We will move the contours of integration to the left, and then the residues that are obtained will be part of the main term in our formula for $ \mathscr{D}_{2,2;\omega}(K) $. Once we obtain the whole main term in terms of $a$ and $b$, we will first let $b$ tend to $a$, and then let $a$ tend to $0$. The resulting limit will provide us with the result of Theorem \ref{thm22}.

Note that we will frequently refer to the special functions that were defined in \eqref{eq:fhG} and \eqref{eq:notation}.


\subsection{Computing $\mathcal{M}_{0,\I, \J; \omega}(K)$} 

\begin{prop}\label{prop:M0}
Let $\I=\{a,0\}$ and $\J=\{b,0\}$, and let $\mathcal{M}_{0,\I, \J; \omega}(K)$ be defined by \eqref{M0swaps}. Then we have
    \[ 
     \mathcal{M}_{0,\I, \J; \omega}(K)
    = \widehat{\omega}(0)\big(\mathcal{R}_1(a, b) + \mathcal{R}_1'(a, b)\big)+O\Big(TK^{-\tfrac{1}{2}+2 \delta}\Big),
    \]
where 
    \begin{align*}
    \mathcal{R}_1(a, b)
    &=\left(Y+c_1+\gamma_0\right)\frac{f(a+b)}{a+b}\frac{f(a)}{a}\frac{f(b)}{b}h(a+b)+2\frac{f(a+b)}{a+b}\frac{f(a)}{a}\frac{f(b)}{b}H(a+b)\\&+\left(\frac{f'(a+b)}{a+b}-\frac{f(a+b)}{(a+b)^2}\right)\frac{f(a)}{a}\frac{f(b)}{b}h(a+b)\\&+\frac{f(a+b)}{a+b}\left(\frac{f'(a)}{a}-\frac{f(a)}{a^2}\right)\frac{f(b)}{b}h(a+b)\\&+\frac{f(a+b)}{a+b}\frac{f(a)}{a}\left(\frac{f'(b)}{b}-\frac{f(b)}{b^2}\right)h(a+b),
    \end{align*}
and 
    \begin{align*}
    \mathcal{R}_1'(a, b)
    &=a^{-2}G(-a)K^{-a}\frac{f(b)}{b}\frac{f(b-a)}{b-a}f(-a)h(b-a)\\&+b^{-2}G(-b)K^{-b}\frac{f(a)}{a}\frac{f(a-b)}{a-b}f(-b)h(a-b)\\&+ (a+b)^{-2}G(-a-b)K^{-a-b}\frac{f(-b)}{b}\frac{f(-a)}{a}f(-a-b)h(-a-b).
    \end{align*}
\end{prop}

 \begin{proof}
 By (1.31) and then by (1.28) in \cite{HN}, we can write 
     \begin{align*}
      \mathcal{M}_{0,\I, \J; \omega}(K)
      &= \frac{\widehat{\omega}(0)}{2\pi i}\int_{(2c)}K^{s}\Phi_2(s)\mathcal{B}\left(\I_{s},\J\right)\;ds\;ds\\
    &= \frac{\widehat{\omega}(0)}{2\pi i}\int_{(2c)}K^{s}\Phi_2(s)\frac{\zeta(1+a+b+s)\zeta(1+a+s)\zeta(1+b+s)\zeta(1+s)}{\zeta(2+2s+a+b)}\;ds
    \end{align*}
We move the line of integration to $\Re(s)=-\frac{1}{2}+2\delta$ capturing the residue of the integrand at $s=0$ in addition to the residues at $s=-a,-b,-a-b$. This gives 
    \begin{align*}
     \mathcal{M}_{0,\I, \J; \omega}(K)
     &=\widehat{\omega}(0)\mathrm{Res}_{s=0}\left(s^{-2}G(s)K^{s}\frac{\zeta(1+a+b+s)\zeta(1+a+s)\zeta(1+b+s)f(s)}{\zeta(2+2s+a+b)}\right)
     \\
     &+ \widehat{\omega}(0)\mathrm{Res}_{s=-a}\left((s^{-2}G(s)K^{s}\frac{\zeta(1+a+b+s)\zeta(1+a+s)\zeta(1+b+s)f(s)}{\zeta(2+2s+a+b)}\right)
     \\
     &+ \widehat{\omega}(0)\mathrm{Res}_{s=-b}\left(s^{-2}G(s)K^{s}\frac{\zeta(1+a+b+s)\zeta(1+a+s)\zeta(1+b+s)f(s)}{\zeta(2+2s+a+b)}\right)
     \\
     &+\widehat{\omega}(0)\mathrm{Res}_{s=-a-b}\left(s^{-2}G(s)K^{s}\frac{\zeta(1+a+b+s)\zeta(1+a+s)\zeta(1+b+s)f(s)}{\zeta(2+2s+a+b)}\right)
     \\
     &+\frac{\widehat{\omega}(0)}{2\pi i}\int_{(-\frac{1}{2}+2\delta)} \Phi_2(s)K^{s}\frac{\zeta(1+a+b+s)\zeta(1+a+s)\zeta(1+b+s)\zeta(1+s)}{\zeta(2+2s+a+b)}\;ds.
    \end{align*}
It follows from (\ref{Phibd}) that 
    \[
   \int_{(-\frac{1}{2}+2\delta)} \Phi_2(s)K^{s}\frac{\zeta(1+a+b+s)\zeta(1+a+s)\zeta(1+b+s)\zeta(1+s)}{\zeta(2+2s+a+b)}\;ds\ll K^{-\frac{1}{2}+2 \delta}.
    \] 
Let us now compute the residue of $s^{-2}G(s)K^{s}\frac{\zeta(1+a+b+s)\zeta(1+a+s)\zeta(1+b+s)f(s)}{\zeta(2+2s+a+b)}$ at $s=0$. This is
    \begin{align*}
    &Y\frac{f(a+b)}{a+b}\frac{f(a)}{a}\frac{f(b)}{b}h(a+b)+
    c_1\frac{f(a+b)}{a+b}\frac{f(a)}{a}\frac{f(b)}{b}h(a+b)\\
    &+\left(\frac{f'(a+b)}{a+b}-\frac{f(a+b)}{(a+b)^2}\right)\frac{f(a)}{a}\frac{f(b)}{b}h(a+b)+\frac{f(a+b)}{a+b}\left(\frac{f'(a)}{a}-\frac{f(a)}{a^2}\right)\frac{f(b)}{b}h(a+b)\\
    &+\frac{f(a+b)}{a+b}\frac{f(a)}{a}\left(\frac{f'(b)}{b}-\frac{f(b)}{b^2}\right)h(a+b)+\gamma_0\frac{f(a+b)}{a+b}\frac{f(a)}{a}\frac{f(b)}{b}h(a+b)\\
    &+2\frac{f(a+b)}{a+b}\frac{f(a)}{a}\frac{f(b)}{b}H(a+b).
    \end{align*}
Further, this is equal to
    \begin{align*}
    &\left(Y+c_1+\gamma_0\right)\frac{f(a+b)}{a+b}\frac{f(a)}{a}\frac{f(b)}{b}h(a+b)+2\frac{f(a+b)}{a+b}\frac{f(a)}{a}\frac{f(b)}{b}H(a+b)\\&+\left(\frac{f'(a+b)}{a+b}-\frac{f(a+b)}{(a+b)^2}\right)\frac{f(a)}{a}\frac{f(b)}{b}h(a+b)\\&+\frac{f(a+b)}{a+b}\left(\frac{f'(a)}{a}-\frac{f(a)}{a^2}\right)\frac{f(b)}{b}h(a+b)\\&+\frac{f(a+b)}{a+b}\frac{f(a)}{a}\left(\frac{f'(b)}{b}-\frac{f(b)}{b^2}\right)h(a+b),
    \end{align*}
which is $\mathcal{R}_1(a, b)$. The desired result is obtained by simply observing that
    \begin{align*}
    \mathcal{R}'_1(a, b)
    &=\mathrm{Res}_{s=-a}\left(s^{-2}G(s)K^{s}\frac{\zeta(1+a+b+s)\zeta(1+a+s)\zeta(1+b+s)f(s)}{\zeta(2+2s+a+b)}\right)\\
    &+ \mathrm{Res}_{s=-b}\left(s^{-2}G(s)K^{s}\frac{\zeta(1+a+b+s)\zeta(1+a+s)\zeta(1+b+s)f(s)}{\zeta(2+2s+a+b)}\right)\\
    &+\mathrm{Res}_{s=-a-b}\left(s^{-2}G(s)K^{s}\frac{\zeta(1+a+b+s)\zeta(1+a+s)\zeta(1+b+s)f(s)}{\zeta(2+2s+a+b)}\right).
    \end{align*}
\end{proof}
We will now rewrite $\mathcal{R}_1(a, b)$ whereby we simplify its expression. For this, we introduce some notation.
    \begin{align*}
    \mathcal{L}_0  &:= Y+c_1+g_1\\
      \kappa_{11}(a,b) & := f(a)f(b) f(a+b) h(a+b),    \\
      \tilde{\kappa}_{11}(a,b)  &: =  f(a)f(b) f(a+b)   H(a+b), \\
      \kappa_{12}(a,b) & := f(a)f(b) h(a+b)  \big((a+b)f'(a+b)-f(a+b)\big) 
        = f(a)f(b) F(a+b)h(a+b), \\
        \kappa_{13}(a,b) & := f(b)f(a+b) h(a+b)  \big(af'(a)-f(a)\big) 
         =  F(a) f(b)f(a+b) h(a+b),  \\
          \kappa_{14}(a,b) &:=  f(a)  f(a+b)  h(a+b)  \big(bf'(b)-f(b)\big)
           =  f(a) F(b) f(a+b)  h(a+b).
      \end{align*} 
Observe that we can now write 
    \[
      \mathcal{R}_1 (a, b)= \big(\mathcal{R}_{11} +  \mathcal{R}_{12} +\mathcal{R}_{13} +\mathcal{R}_{14}\big)(a, b),
    \]
where we set
    \begin{equation}\label{eqn:R1}
    \begin{split}
       \mathcal{R}_{11} (a, b)& = \frac{1}{ab(a+b)} \mathcal{L}_0 \kappa_{11}(a,b) 
       +2  \frac{1}{ab(a+b)}  \tilde{\kappa}_{11}(a,b) , \\
        \mathcal{R}_{12}(a, b) & =  \frac{1}{ab(a+b)^2} \kappa_{12}(a,b), \\
        \mathcal{R}_{13} (a, b)& =  \frac{1}{a^2b(a+b)} \kappa_{13}(a,b),\\
        \mathcal{R}_{14} (a, b)& =   \frac{1}{ab^2(a+b)} \kappa_{14}(a,b).
    \end{split}
    \end{equation}


\subsection{Computing $\mathcal{M}_{1,\I, \J; \omega}(K)$}
 First, we observe that by \eqref{M1swaps} we have 
  \begin{align*}
  \mathcal{M}_{1,\I, \J; \omega}(K)=\sum_{i_1\in\K,i_2\in\Ll}\frac{c_{i_{1},i_{2}}}{2\pi i}\int_{-\infty}^{\infty}\omega(t)\int_{\Re(s)=2\epsilon}I_{i_1i_2}(s,t)\;ds\;dt
  \end{align*} 
  for sufficiently small $\epsilon>0$, where 
    \begin{align*}      
      c_{i_{1},i_{2}} 
      &= 
   \Z\big(\I \backslash \{ a_{i_1} \}, \{ -a_{i_1} \}\big) \Z\big( \{ - b_{i_2} \}, \J \backslash \{ -b_{i_2} \}\big) \\
   &= \prod_{j_1 \in \K \setminus \{ i_1 \}}  \zeta(1-a_{i_1}+a_{j_1})   
      \prod_{j_2 \in  \Ll \setminus \{ i_2 \} }
    \zeta(1-b_{i_2}+b_{j_2}),
   \end{align*}
 and 
    \begin{align*}
    I_{i_1i_2}(s,t)
    &=\Phi_2(s) K^{s} \left( \frac{2\pi}{t}\right)^{a_{i_1}+b_{i_2}+s}\zeta(1-a_{i_1}-b_{i_2}-s)
    \prod_{\substack{j_1\in \K\backslash\{{i_1}\}\\j_2\in \Ll\backslash\{{i_2}\}}}\zeta(1+a_{j_1}+b_{j_2}+s)\\
    &\hspace{2em} \times\mathcal{A}\big( (\I \setminus \{ a_{i_1} \}) \cup \{ -b_{i_2}-s \},  ((\J \setminus \{ b_{i_2} \})+s)  \cup \{ - a_{i_1} \} \big).
    \end{align*}
Since we chose  $\I=\{a,0\}$ and $\J=\{b,0\}$, the terms $c_{i_1,i_2}$ and $I_{i_1,i_2}(s,t)$ that appear in $\mathcal{M}_{1,\I,\J;\omega}(K)$ can be written more explicitly. We find that 
    \begin{align*}
     I_{11}(s,t) 
     & = \Phi_2(s) K^{s} \Big( \frac{2 \pi}{t} \Big)^{a+b+s} \zeta(1-a-b-s) \zeta(1+s)
     \mathcal{A}\big( \{ 0, -b-s\}, \{  s, -a \}\big) \\
     & = \frac{G(s)}{s} K^{s} \Big( \frac{2 \pi}{t} \Big)^{a+b+s} \zeta(1-a-b-s) \zeta(1+s) \frac{1}{\zeta(2-a-b)} ,
     \end{align*}
    \begin{align*}
     I_{12}(s,t) 
     & = \Phi_2(s) K^{s} \Big( \frac{2 \pi}{t} \Big)^{a+s} \zeta(1-a-s) \zeta(1+b+s)
     \mathcal{A}\big( \{ 0, -s\}, \{  b+ s, -a \}\big) \\
     & =  \frac{G(s)}{s} K^{s} \Big( \frac{2 \pi}{t} \Big)^{a+s} \zeta(1-a-s) \zeta(1+b+s)
     \frac{1}{\zeta(2+ b-a)},
    \end{align*}
    \begin{align*}
     I_{21}(s,t)
     &= \Phi_2(s) K^{s} \Big( \frac{2 \pi}{t} \Big)^{b+s} \zeta(1-b-s) \zeta(1+a+s)
     \mathcal{A}\big( \{ a, -b-s\}, \{  s, 0 \}\big) \\
     & =   \frac{G(s)}{s} K^{s} \Big( \frac{2 \pi}{t} \Big)^{b+s} \zeta(1-b-s) \zeta(1+a+s)
     \frac{1}{\zeta(2+a-b)},
    \end{align*}
and 
    \begin{align*}
     I_{22}(s,t) 
     & =\Phi_2(s) K^{s} \Big( \frac{2 \pi}{t} \Big)^{s} \zeta(1-s) \zeta(1+a+b+s)
     \mathcal{A}\big( \{ a, -s\}, \{  b+s, 0 \}\big) \\
     & =  \frac{G(s)}{s} K^{s} \Big( \frac{2 \pi}{t} \Big)^{s} \zeta(1-s) \zeta(1+a+b+s) \frac{1}{\zeta(2+a+b)}.
     \end{align*}
One can compute $c_{i_1,i_2}$ in a straightforward manner. We collect the results in the following table.
    \begin{table}[ht]
    \caption{}
    \begin{center}
    \setlength{\tabcolsep}{16pt}
        \def\arraystretch{2}
    \begin{tabular}{ |c|c|c| } 
     \hline
     $(i_1,i_2)$ & $c_{i_1,i_2}$ & $I_{i_1i_2}(s,t)$ \\
     \hline
     $(1,1)$ & $\zeta(1-a) \zeta(1-b)$ & $ \frac{1}{s^2} G(s) K^{s} \left( \tfrac{2 \pi}{t} \right)^{a+b+s} \zeta(1-a-b-s) f(s) \frac{1}{\zeta(2-a-b)}$ \\ \hline
     $(1,2)$ & $\zeta(1-a) \zeta(1+b)$ & $ \frac{1}{s}G(s) K^{s} \left( \tfrac{2 \pi}{t} \right)^{a+s} \zeta(1-a-s) \zeta(1+b+s)
     \frac{1}{\zeta(2+ b-a)}$ \\ \hline
     $(2,1)$ &  $ \zeta(1+a) \zeta(1-b)$        & $\frac{1}{s}G(s) K^{s} \left( \tfrac{2 \pi}{t} \right)^{b+s} \zeta(1-b-s) \zeta(1+a+s)
     \frac{1}{\zeta(2+a-b)}$\\\hline
     $(2,2)$ &   $\zeta(1+a) \zeta(1+b)$        & $- \frac{1}{s^2}G(s) K^s \left( \tfrac{2 \pi}{t} \right)^{s} \zeta(1+a+b+s)f(-s) \frac{1}{\zeta(2+a+b)} $ \\\hline
    \end{tabular}\label{table c I}
    \end{center}
    \end{table}   
    
Hence we can write 
    \begin{equation}\label{eqn:offdiag-CK-simplified}
    \begin{split}
     &\mathcal{M}_{1,\I, \J; \omega}(K) \\
    =&\frac{1}{2\pi i} \int_{-\infty}^{\infty}\omega(t)\int_{\Re(s)=2\epsilon}\Big( \zeta(1-a) \zeta(1-b)I_{11}(s,t)+ \zeta(1-a) \zeta(1+b)I_{12}(s,t)\\
    &\hspace{9em}+\zeta(1+a) \zeta(1-b)I_{21}(s,t)+\zeta(1+a) \zeta(1+b)I_{22}(s,t)\Big) \;ds\;dt.
    \end{split}
    \end{equation}


    \begin{prop}\label{prop:M1}
    Let $K=T^{1+\eta}$ with $0 < \eta<\frac13$, and suppose that a weight function $\omega$ satisfies \eqref{cond1}, \eqref{cond2}, and \eqref{cond3} with $\nu>\frac{5+3(\eta+1)}{9}$. Let $\I=\{a,0\}$ and $\J=\{b,0\}$ satisfy \eqref{sizerestriction} and \eqref{sizerestriction-}. In particular, assume that $|a|,|b|\leq\delta$ with $\delta<\frac{\eta}{2(2+3\eta)}$. Then we have       
    \begin{equation*}
       \mathcal{M}_{1,\I, \J; \omega}(K) = \int_{-\infty}^{\infty} \omega(t) \cdot  \big( -\mathcal{R}_1'(a, b)+\mathcal{R}_{2}(a, b)\big) \, dt +O\left(K^{-\frac12+3\delta}T^{\frac32-\delta}\right),
    \end{equation*}
    where $\mathcal{R}_1'(a, b)$ is as given in Proposition \ref{prop:M0} and 
        \begin{equation}\label{eq:R2}
        \begin{split}
           \mathcal{R}_{2}(a, b)
           =&  -  \left( \frac{2 \pi}{t}\right )^{a+b}h(-a-b)\frac{f(-a)}{a}\frac{f(-b)}{b}
         \bigg(  \frac{F(-a-b)}{(a+b)^2}+  \frac{f(-a-b)}{a+b} (X + g_1 + c_1)
          \bigg) \\
          & +  \left( \frac{2 \pi}{t}\right )^{a} h(b-a) \frac{f^2(-a)}{a^2} \frac{f^2(b)}{b^2} +    \left( \frac{2 \pi}{t} \right)^{b} h(a-b) \frac{f^2(-b)}{b^2}\frac{f^2(a)}{a^2}\\
           & -   h(a+b)\frac{f(a)}{a}\frac{f(b)}{b}
         \left( \frac{F(a+b)}{(a+b)^2}+  \frac{f(a+b)}{a+b} (X - g_1 + c_1)   \right)
        \\
        &  +  K^{-b} \left( \frac{2 \pi}{t} \right)^{a-b} h(b-a)\frac{G(-b)}{b} \frac{f(b-a)}{b-a}\frac{f(-a)}{a}\frac{f(b)}{b}   \\
        & + K^{-a} \left( \frac{2 \pi}{t} \right)^{b-a}h(a-b)\frac{G(-a)}{a}\frac{f(a-b)}{a-b}\frac{f(a)}{a}\frac{f(-b)}{b} \\
          &- K^{-a-b}\Big( \frac{2 \pi}{t} \Big)^{-a-b}  h(a+b)\frac{G(-a-b)}{a+b}  \frac{f(a+b)}{a+b}\frac{f(a)}{a}\frac{f(b)}{b}.
        \end{split}
        \end{equation}
    \end{prop}


\begin{proof}
Observe that by Table \ref{table c I}, each of $I_{11}(s,t)$ and $I_{22}(s,t)$ in \eqref{eqn:offdiag-CK-simplified} have 
    \begin{itemize}
    \item a double pole at $s=0$,  
    \item a simple pole at $s= -(a+b)$, 
    \end{itemize}
whereas  $I_{12}(s,t)$ and $I_{21}(s,t)$ in \eqref{eqn:offdiag-CK-simplified}  each have
    \begin{itemize}
    \item a simple pole at $s=0$,  
    \item a simple pole at $s= -a$, 
    \item a simple pole at $s= -b$.
    \end{itemize}
We denote by $R_{i_1i_2}(a, b)$ the  sum of the residues of $I_{i_1i_2}(s,t)$ at these poles.  Moving the contour of integration in \eqref{eqn:offdiag-CK-simplified} to the line $\Re(s)=-\frac12+3\delta $ gives 
    \begin{equation}\label{eqn:offdiag-residues+int}
    \begin{split}
    \mathcal{M}_{1,\I, \J; \omega}(K)
    =\, & \int_{-\infty}^{\infty}\omega(t)\Big(\zeta(1-a) \zeta(1-b)R_{11}(a, b)+ \zeta(1-a) \zeta(1+b)R_{12}(a, b) \\
    &\hspace{5em}+\zeta(1+a) \zeta(1-b)R_{21}(a, b)+\zeta(1+a) \zeta(1+b)R_{22} (a, b)\Big)\;dt  \\
    &+\sum_{i_1\in\K,i_2\in\Ll}\frac{c_{i_{1},i_{2}}}{2\pi i}\int_{-\infty}^{\infty}\omega(t)\int_{(-\frac12+3\delta)}I_{i_1i_2}(s,t)\;ds\;dt.
    \end{split}
    \end{equation} 
We first estimate the second term on the right-hand side, which is equal to
    \begin{align*}
    &\sum_{i_{1}\in\K,i_2\in\Ll}
    c_{i_{1},i_{2}}\frac{1}{2\pi i}\int_{(-\frac12+3\delta)}\mathcal{A}( (\I \setminus \{ a_{i_1} \}) \cup \{ -b_{i_2}-s \},   
         ((\J \setminus \{ b_{i_2} \})+s)
          \cup \{ - a_{i_1} \} )K^{s}\Phi_2(s) \\
     &\hspace{1em}\times \zeta(1-a_{i_1}-b_{i_2}-s)\prod_{\substack{{j_1}\in \K\backslash\{{i_1}\}\\{j_2}\in \Ll\backslash\{{i_2}\}}}\zeta(1+a_{j_1}+b_{j_2}+s)\int_{-\infty}^{\infty}\left(\frac{t}{2\pi}\right)^{-s-a_{i_1}-b_{i_2}}\omega(t)\;dt\;ds.
    \end{align*}
By using $ |\zeta(\sigma+it)| \ll  t^{\tfrac{1-\sigma}{2}} \log t$ for $\sigma \in (0,1)$ and $ |\zeta(\sigma+it)| \ll 1$ for $\sigma \in [1.01,2]$, we observe that for $s=-\frac12+3\delta+iu$, we have  
    \begin{align*}
    \zeta(1-a_{i_1}-b_{i_2}-s)\prod_{\substack{{j_1}\in \K\backslash\{{i_1}\}\\{j_2}\in \Ll\backslash\{{i_2}\}}}\zeta(1+a_{j_1}+b_{j_2}+s)
    \ll \Big((|u|+1)^{\frac14-\frac{\delta}{2}}\log(2+|u|)\Big)^{(k-1)(\ell-1)}.
    \end{align*}
We also know by \cite[Proposition~5.2]{HN} that $\mathcal{A}( (\I \setminus \{ a_{i_1} \}) \cup \{ -b_{i_2}-s \},   ((\J \setminus \{ b_{i_2} \})+s)  \cup \{ - a_{i_1} \} )= O(1)$ when $\Re(s)\geq-1+2\delta+\epsilon$. It follows that  
      \begin{align*}
    &\int_{(-\frac12+3\delta)}\mathcal{A}( (\I \setminus \{ a_{i_1} \}) \cup \{ -b_{i_2}-s \},   
         ((\J \setminus \{ b_{i_2} \})+s)
          \cup \{ - a_{i_1} \} )K^{s}\Phi_2(s) \\&\hspace{1em}\times \zeta(1-a_{i_1}-b_{i_2}-s)\prod_{\substack{{j_1}\in \K\backslash\{{i_1}\}\\{j_2}\in \Ll\backslash\{{i_2}\}}}\zeta(1+a_{j_1}+b_{j_2}+s)\left(\frac{t}{2\pi}\right)^{-s}\;ds 
          \ll K^{-\frac12+3\delta}t^{\frac12-3\delta}.
    \end{align*}
Therefore,
\begin{equation}\label{eqn:off-diag-error}
\begin{split}
&\sum_{i_{1}\in\K,i_2\in\Ll}
c_{i_{1},i_{2}}\frac{1}{2\pi i}\int_{(-\frac12+3\delta)}\mathcal{A}( (\I \setminus \{ a_{i_1} \}) \cup \{ -b_{i_2}-s \},   
     ((\J \setminus \{ b_{i_2} \})+s)
      \cup \{ - a_{i_1} \} )K^{s}\Phi_2(s) \\&\hspace{1em}\times \zeta(1-a_{i_1}-b_{i_2}-s)\prod_{\substack{{j_1}\in \K\backslash\{{i_1}\}\\{j_2}\in \Ll\backslash\{{i_2}\}}}\zeta(1+a_{j_1}+b_{j_2}+s)\int_{-\infty}^{\infty}\left(\frac{t}{2\pi}\right)^{-s-a_{i_1}-b_{i_2}} \;ds\; \omega(t) dt\\
      &\ll K^{-\frac12+3\delta} \int_{-\infty}^{\infty} \omega(t)t^{\frac12-\delta}\;dt
      \ll K^{-\frac12+3\delta}T^{\frac32-\delta}.
\end{split}
\end{equation}
Note that since $K=T^{1+\eta}$, we require $\delta<\tfrac{\eta}{2(2+3\eta)}$. 

Next, we compute the terms $R_{11}(a, b)$, $R_{12}(a, b)$, $R_{21}(a, b)$ and $R_{22}(a, b)$ in \eqref{eqn:offdiag-residues+int}. We have
    $$
    R_{11} (a, b)
    = \text{Res}_{s=0} \left( I_{11}(s) \right) + \text{Res}_{s=-a-b} \left( I_{11}(s) \right).
    $$
For the first residue, we have
    \begin{equation*}
    \text{Res}_{s=0} \left( I_{11}(s) \right)  =  \text{Res}_{s=0} \left( \frac{U(s)}{s^2} \right) = U'(0),
    \end{equation*}
where 
    \[
      U(s) =  \Big( \frac{2 \pi}{t} \Big)^{a+b}  \frac{1}{\zeta(2-a-b)} \left(\frac{K}{\tfrac{t}{2\pi}}\right)^s \zeta(1-a-b-s) f(s) G(s).
    \]
Since $\displaystyle X=\log \Big(K/\tfrac{t}{2\pi}\Big)$, we have 
    \begin{equation*}
    \begin{split}
     U'(0)  =   \frac{  ( \frac{2 \pi}{t} )^{a+b}}{\zeta(2-a-b)}
      \Big(&
     X \zeta(1-a-b) f(0) G(0) - \zeta'(1-a-b) f(0)G(0) \\
     &  +  \zeta(1-a-b) f'(0) G(0)
     + \zeta(1-a-b) f(0) G'(0)
     \Big). 
    \end{split}
    \end{equation*}
It follows that 
    \begin{equation*}
    \text{Res}_{s=0} \left( I_{11}(s) \right) 
     = U'(0) 
      =   \frac{  ( \tfrac{2 \pi}{t} )^{a+b}}{\zeta(2-a-b)}
     \Big( - \zeta'(1-a-b)+  \zeta(1-a-b) (X + g_1 + c_1) \Big).  
    \end{equation*}
Since $s=-(a+b)$ is a simple pole, we have 
    \[
     \text{Res}_{s=-a-b} \left( I_{11}(s) \right) = - \Phi_2(-a-b) K^{-a-b} \zeta(1-a-b)  \frac{1}{\zeta(2-a-b)}.
    \]
Thus we obtain
    \begin{equation}\label{eqn:R11}
    \begin{split}
    R_{11}(a, b) 
     =\,& \frac{  \left( \frac{2 \pi}{t} \right)^{a+b}}{\zeta(2-a-b)}
     \Big( - \zeta'(1-a-b)+  \zeta(1-a-b) (X + g_1 + c_1)
      \Big) \\
      & - \Phi_2(-a-b) K^{-a-b} \zeta(1-a-b)  \frac{1}{\zeta(2-a-b)}\\
      =\,& - \left( \frac{2 \pi}{t}\right )^{a+b}h(-a-b)
     \bigg(  \frac{F(-a-b)}{(a+b)^2}+  \frac{f(-a-b)}{a+b} (X + g_1 + c_1)
      \bigg) \\
      &- K^{-a-b}h(-a-b)\frac{G(-a-b)}{a+b} \frac{f(-a-b)}{a+b} .
    \end{split}
    \end{equation}
Next for $ R_{22} $ , we note that
    $$
    R_{22}(a, b) = \text{Res}_{s=0} \left( I_{22}(s) \right) + \text{Res}_{s=-a-b} \left( I_{22}(s) \right).
    $$
Here
    \[
     \text{Res}_{s=0} \left( I_{22}(s) \right) =  \text{Res}_{s=0} \left( \frac{V(s)}{s^2}\right) = U'(0),
    \]
where 
    \[
      V(s) =  - \frac{1}{\zeta(2+a+b)} \left(\frac{K}{\tfrac{t}{2\pi}}\right)^s \zeta(1+a+b+s) f(-s) G(s).
    \]
We compute
    \begin{align*}
     V'(0)  =  - \frac{1}{\zeta(2+a+b)}
     \Big(
     & X \zeta(1+a+b) f(0) G(0) + \zeta'(1+a+b) f(0)G(0) \\
     &  +  \zeta(1+a+b) (-1) f'(0) G(0)
     + \zeta(1+a+b) f(0) G'(0)
     \Big). 
    \end{align*}
It then follows that 
    \[
     \text{Res}_{s=0} \left( I_{22}(s) \right) = V'(0)  =  - \frac{1}{\zeta(2+a+b)}
     \left( \zeta'(1+a+b)+  \zeta(1+a+b) (X - g_1 + c_1)
      \right).  
    \]
For the other residue, since $s=-(a+b)$ is a simple pole we have 
    \[
     \text{Res}_{s=-a-b} \left( I_{22}(s) \right) = \Phi_2(-a-b) K^{-a-b} \Big( \frac{2 \pi}{t} \Big)^{-a-b} \zeta(1+a+b)  \frac{1}{\zeta(2+a+b)}.
    \]
Hence
    \begin{equation}\label{eqn:R22}
    \begin{split}
    R_{22}(a, b)  =\,&- \frac{1}{\zeta(2+a+b)}
     \left( \zeta'(1+a+b)+  \zeta(1+a+b) (X - g_1 + c_1)
      \right)
     \\&+  \Phi_2(-a-b) K^{-a-b} \Big( \frac{2 \pi}{t} \Big)^{-a-b} \zeta(1+a+b)  \frac{1}{\zeta(2+a+b)}\\
     =\, & -h(a+b)
     \left( \frac{F(a+b)}{(a+b)^2}+  \frac{f(a+b)}{a+b} (X - g_1 + c_1)
      \right)
     \\
      &+  K^{-a-b}\Big( \frac{2 \pi}{t} \Big)^{-a-b}  h(a+b)\frac{G(-a-b)}{-a-b}  \frac{f(a+b)}{a+b}.
    \end{split}
    \end{equation}
It remains to compute $R_{12}$ and $R_{21}$. We have
    $$
    R_{12}(a, b) 
    = \text{Res}_{s=0} \left( I_{12}(s) \right) + \text{Res}_{s=-a} \left( I_{12}(s) \right) + \text{Res}_{s=-b} \left( I_{12}(s) \right)
    $$
and 
    $$
    R_{21}(a, b)
    = \text{Res}_{s=0} \left( I_{21}(s)\right) + \text{Res}_{s=-a} \left( I_{21}(s) \right) + \text{Res}_{s=-b} \left( I_{21}(s)\right).
    $$
For $R_{12}(a, b)$, we note that
    \begin{align*}
    \text{Res}_{s=0} \left( I_{12}(s) \right) & =  \left( \frac{2 \pi}{t} \right)^{a}  \zeta(1-a) \zeta(1+b)  \frac{1}{\zeta(2+b-a)}, \\
     \text{Res}_{s=-a} \left( I_{12}(s) \right) & = - \Phi_2(-a) K^{-a}\zeta(1+b-a)  \frac{1}{\zeta(2+b-a)}, \\
\text{and} \quad \quad   \text{Res}_{s=-b} \left( I_{12}(s) \right) & =  \Phi_2(-b) K^{-b} \left( \frac{2 \pi}{t}\right )^{a-b} \zeta(1+b-a)  \frac{1}{\zeta(2+b-a)},
    \end{align*}
so
    \begin{equation}\label{eqn:R12}
    \begin{split}
      R_{12} (a, b)
      =\,  &  -\left( \frac{2 \pi}{t}\right )^{a} h(b-a) \frac{f(-a)}{a} \frac{f(b)}{b}   \\
      & + K^{-a}h(b-a) \frac{G(-a)}{a} \frac{f(b-a)}{b-a}   -  K^{-b} \left( \frac{2 \pi}{t} \right)^{a-b} h(b-a)\frac{G(-b)}{b} \frac{f(b-a)}{b-a}  .
    \end{split}
    \end{equation}
For $R_{21}(a, b) $, we will use
    \begin{equation*}
    \begin{split}
     \text{Res}_{s=0} \left( I_{21}(s)\right)  & = \left( \frac{2 \pi}{t}\right)^{b}  \zeta(1-b) \zeta(1+a)  \frac{1}{\zeta(2+a-b)}, \\
     \text{Res}_{s=-a} \left( I_{21}(s) \right) & = -  \Phi_2(-b) K^{-b}\zeta(1+a-b)  \frac{1}{\zeta(2+a-b)}, \\
     \text{Res}_{s=-b} \left( I_{21}(s)\right) & =  \Phi_2(-a) K^{-a} \left( \frac{2 \pi}{t} \right)^{b-a} \zeta(1-b+a)  \frac{1}{\zeta(2+a-b)}, \\
    \end{split}
    \end{equation*}
and find that
    \begin{equation}\label{eqn:R21}
    \begin{split}
      R_{21}(a, b) & =  -\left( \frac{2 \pi}{t} \right)^{b} h(a-b) \frac{f(-b)}{b}\frac{f(a)}{a}  \\
       & +K^{-b}h(a-b)\frac{G(-b)}{b}  \frac{f(a-b)}{a-b}   -K^{-a} \left( \frac{2 \pi}{t} \right)^{b-a}h(a-b)\frac{G(-a)}{a}\frac{f(a-b)}{a-b}. \\
    \end{split}
    \end{equation}
Inserting \eqref{eqn:off-diag-error}, \eqref{eqn:R11}, \eqref{eqn:R22}, \eqref{eqn:R12},  \eqref{eqn:R21} in \eqref{eqn:offdiag-residues+int} yields the desired result.
\end{proof}

Now we will rewrite $\mathcal{R}_2(a, b)$ in order to simplify its expression. We set  
\[
\mathcal{L}' = X + g_1 + c_1 \quad \text{and} \quad \mathcal{L}'' = X - g_1 + c_1,
\]
and also
\begin{equation}\label{eqn:kappas}
\begin{split}
    \kappa_{25}(a,b) & = - E_1(-a-b) h(-a-b) f(-a)f(-b)F(-a-b),  \\
    \tilde{\kappa}_{25}(a,b) &  = E_1(-a-b)  h(-a-b)f(-a)f(-b)f(-a-b)  , \\
     \kappa_{26}(a,b) & =E_1(-a) h(b-a) f(-a)^2 f(b)^2  ,  \\
      \kappa_{27}(a,b) & = E_1(-b)h(a-b) f(a)^2 f(-b)^2  ,   \\
       \kappa_{28}(a,b)  & =h(a+b) f(a)f(b)((a+b)f'(a+b)-f(a+b)) , \\
        \tilde{\kappa}_{28}(a,b)& =h(a+b) f(a) f(b) f(a+b) , \\
           \kappa_{29}(a,b)  & = E_2(-b)E_1(b-a) h(b-a) G(-b)f(-a)f(b)  f(b-a) ,\\
              \kappa_{210}(a,b)  & =E_2(-a)E_1(a-b) h(a-b)  G(-a) f(a)f(-b)f(a-b)  ,\\
   \kappa_{211}(a,b)  & =  E_2(-a-b) E_1(a+b) h(a+b)  f(a)f(b) G(-a-b) f(a+b) . 
   \end{split}
\end{equation} 
With this notation and by \eqref{eq:R2}, we can write 
\begin{equation*}
 \mathcal{R}_2(a, b)= \big(\mathcal{R}_{25} + \mathcal{R}_{26} +  \mathcal{R}_{27} +  \mathcal{R}_{28} 
 +  \mathcal{R}_{29} +  \mathcal{R}_{210} +  \mathcal{R}_{211}\big)(a, b),
\end{equation*}
where we set
\begin{equation}\label{eqn:R2}
\begin{split}
      \mathcal{R}_{25}(a, b)  & = \frac{1}{ab(a+b)^2} \kappa_{25}(a,b)- \frac{1}{ab(a+b)} \tilde{\kappa}_{25}(a,b) \mathcal{L}', \\
        \mathcal{R}_{26} (a, b) & =  \frac{1}{(ab)^2} \kappa_{26}(a,b), \\
          \mathcal{R}_{27}(a, b)  & = \frac{1}{(ab)^2}  \kappa_{27}(a,b),  \\  
          \mathcal{R}_{28} (a, b) & = - \frac{1}{ab(a+b)^2} \kappa_{28}(a,b)
           - \frac{1}{ab(a+b)}  \tilde{\kappa}_{28}(a,b) \mathcal{L}'', \\
            \mathcal{R}_{29}  (a, b)& = \frac{1}{ab^2(b-a)} \kappa_{29}(a,b), \\
              \mathcal{R}_{210} (a, b) & = \frac{1}{a^2 b(a-b)} \kappa_{210}(a,b), \\
                \mathcal{R}_{211} (a, b) & = - \frac{1}{a b(a+b)^2} \kappa_{211}(a,b).
\end{split}
\end{equation}
By Theorem \ref{HN-main}, Proposition \ref{prop:M0} and Proposition \ref{prop:M1} we arrive at the following proposition.

\begin{prop}
Let $\I=\{a, 0\} $ and $\J=\{b, 0\} $. Then
\begin{equation*}
\mathscr{D}_{\I,\J;\omega}(K)
=\int_{-\infty}^{\infty} \omega(t) \cdot  \mathcal{R}(a, b) \, dt +O\bigg(T^{\frac34(1+\eta)+\e}\Big(\frac{T}{T_0}\Big)^{\frac94}+T^{1-\frac{\eta}{2}}\bigg),
\end{equation*}
where
\begin{align*}
      \mathcal{R}(a, b)   =  \frac{1}{ab} \Bigg(& \frac{1}{(a+b)} (L+2 g_1) \kappa_{11}(a,b) 
   +2  \frac{1}{(a+b)}  \tilde{\kappa}_{11}(a,b)   +  \frac{1}{a(a+b)} \kappa_{13}(a,b) \\
   &+   \frac{1}{b(a+b)} \kappa_{14}(a,b)  
      +   \frac{1}{(a+b)^2} \kappa_{25}(a,b)- \frac{1}{(a+b)} \tilde{\kappa}_{25}(a,b) \mathcal{L}' +
       \frac{1}{ab} \kappa_{26}(a,b) \\
       &+ \frac{1}{ab}  \kappa_{27}(a,b)   
        + 
         \frac{1}{b(b-a)} \kappa_{29}(a,b)+ \frac{1}{a (a-b)} \kappa_{210}(a,b)-  \frac{1}{(a+b)^2} \kappa_{211}(a,b)
        \Bigg). 
\end{align*}
\end{prop}

\begin{proof}
We have $\mathcal{R}(a, b)=\mathcal{R}_1(a, b)+\mathcal{R}_{2}(a, b)$. The result follows from  \eqref{eqn:R1}, \eqref{eqn:R2}, and the observations that $\kappa_{12}(a,b) = \kappa_{28}(a,b)$, $\kappa_{11}(a,b)  =\tilde{\kappa}_{28}(a,b)$, and 
 $ \mathcal{L}_0-\mathcal{L}'' = \log \tfrac{t}{2 \pi} + 2g_1 = L +2 g_1$. 
 \end{proof}



\subsection{Computing $\mathcal{R}(a,a)=\lim_{a,b\to0}\mathcal{R}(a,b)$} Our goal is now reduced to computing the limit of $\mathcal{R}(a, b)$ as $a,b\to0$. To this end, we write down the Taylor series expansions of the entire functions $\kappa_{1*}$, $\kappa_{2*}$ and $\tilde{\kappa}_{2*}$ using \eqref{eq:notation} and \eqref{eqn:kappas}, and then we combine the terms with similar coefficients to obtain the expression 
\begin{equation}\label{eq:defn R}
  \mathcal{R}(a, b) = \left(A_1 + \tilde{A}_1 + A_2 + A_3 + A_4 + A_5 +A_6\right)(a,b),
\end{equation}
where the functions $A_1, \tilde{A}_1, A_2, A_3, A_4, A_5$, and $A_6$  are given as follows:
\begin{align*} 
 A_1  (a,b)
 =\, &  \frac{1}{ab(a+b)} (L+2 g_1) \kappa_{11}(a,b) 
    +  \frac{1}{a^2b(a+b)} \kappa_{13}(a,b)+   \frac{1}{ab^2(a+b)} \kappa_{14}(a,b)   
 \\
 =\, &\frac{1}{ab(a+b)}
 \sum_{j_1,j_2,j_3} g_{j_1} g_{j_2} (g \star \delta)_{j_3} 
 (L+2g_1) a^{j_1} b^{j_2} (a+b)^{j_3} \\
 &+\frac{1}{a^2b(a+b)}
 \sum_{j_1,j_2,j_3} g_{j_1} g_{j_2} (g \star \delta)_{j_3} 
 (j_1-1)a^{j_1} b^{j_2} (a+b)^{j_3}  \\
 &+\frac{1}{ab^2(a+b)}
 \sum_{j_1,j_2,j_3} g_{j_1} g_{j_2} (g \star \delta)_{j_3} 
  (j_2-1) a^{j_1} b^{j_2} (a+b)^{j_3} , \\
  \tilde{A}_1 (a,b) =\, &   \frac{2}{ab(a+b)}  \tilde{\kappa}_{11}(a,b)    
 =\frac{2}{ab(a+b)}  \sum_{j_1,j_2,j_3}  g_{j_1} g_{j_2} (g \star \delta')_{j_3} 
   a^{j_1} b^{j_2} (a+b)^{j_3} , \\
 A_2 (a,b)
  =\,&  \frac{1}{(ab)^2} \kappa_{26}(a,b) + \frac{1}{(ab)^2}  \kappa_{27}(a,b)   \\
  =\, & \frac{1}{(ab)^2} \sum_{j_1,j_2,j_3} (-1)^{j_1} 
  (\alpha*g*g)_{j_1} (g*g)_{j_2} \delta_{j_3} \Big\{  a^{j_1} b^{j_2} + (-1)^{j_3} a^{j_2} b^{j_1}\Big\} (b-a)^{j_3}, \\
 A_3 (a,b)
  =\, & \frac{1}{ab^2(b-a)} \kappa_{29}(a,b)+ \frac{1}{a^2b (a-b)} \kappa_{210}(a,b)   \\
 =\, &\frac{1}{ab^2(b-a)} \sum_{j_1, j_2, j_3}  
      g_{j_1} (c \star \beta \star (-1)^{\bullet} g)_{j_2}
       (g \star \alpha \star \delta)_{j_3}  (-1)^{j_1+j_2}
      a^{j_1} b^{j_2}  (b-a)^{j_3}\\
      &-\frac{1}{a^2b(b-a)} \sum_{j_1, j_2, j_3}  
      g_{j_1} (c \star \beta \star (-1)^{\bullet} g)_{j_2}
       (g \star \alpha \star \delta)_{j_3}  (-1)^{j_1+j_2}
        (-1)^{j_3}  a^{j_2} b^{j_1}   (b-a)^{j_3}, \\
 A_4(a,b)  
 =\, &-  \frac{1}{ab(a+b)^2} \kappa_{211}(a,b) \\
 =\, & -\frac{1}{ab(a+b)^2}   \sum_{j_1,j_2,j_3} g_{j_1} g_{j_2} (g \star \delta \star \alpha \star (-1)^{\bullet} c \star
     (-1)^{\bullet} \beta)_{j_3}  a^{j_1} b^{j_2} (a+b)^{j_3} , \\
 A_5 (a,b)
  =\, & \frac{1}{ab(a+b)^2} \kappa_{25}(a,b)
 = -\frac{1}{ab(a+b)^2} 
 \sum_{j_1, j_2, j_3} g_{j_1} g_{j_2}   (\alpha \star g' \star \delta)_{j_3} (-1)^{j_1+j_2+j_3}  a^{j_1} b^{j_2} (a+b)^{j_3}, \\
 A_6  (a,b)
 =\,& - \frac{\mathcal{L}' }{ab(a+b)} \tilde{\kappa}_{25}(a,b) 
 =  -\frac{\mathcal{L}'}{ab(a+b)}   \sum_{j_1,j_2, j_3} g_{j_1} g_{j_2} (g \star \alpha \star \delta)_{j_3} (-1)^{j_1+j_2+j_3}  a^{j_1} b^{j_2} (a+b)^{j_3}.
\end{align*}
 We will first compute $\lim_{b\to a}\mathcal{R}(a, b)$ and then use Maple to find $\lim_{a\to0} \left(\lim_{b\to a} \mathcal{R}(a, b)\right)$. It is straightforward to see that 
    \begin{equation}\label{eq:A1-A6}
    \begin{split}
     \lim_{b\to a}  A_1  (a,b)
    =\, &\frac{1}{2a^3}
     \sum_{j_1,j_2,j_3} g_{j_1} g_{j_2} (g \star \delta)_{j_3} 
     (L+2g_1)  2^{j_3}a^{j_1+j_2+j_3} \\
     &+\frac{1}{2a^4}
     \sum_{j_1,j_2,j_3} g_{j_1} g_{j_2} (g \star \delta)_{j_3} 
     (j_1-1)2^{j_3}a^{j_1+j_2+j_3}    \\
     &+\frac{1}{2a^4}
     \sum_{j_1,j_2,j_3} g_{j_1} g_{j_2} (g \star \delta)_{j_3} 
      (j_2-1) 2^{j_3} a^{j_1+j_2+j_3} ,    \\
     \lim_{b\to a}  \tilde{A}_1   (a,b)
     =\, &\frac{1}{a^3}  \sum_{j_1,j_2,j_3}  g_{j_1} g_{j_2} (g \star \delta')_{j_3} 
       2^{j_3}  a^{j_1+j_2+j_3}  ,   \\
     \lim_{b\to a}A_2 (a,b)  =\, & \frac{2\delta_0}{a^4} \sum_{j_1,j_2} (-1)^{j_1} 
      (\alpha*g*g)_{j_1} (g*g)_{j_2}  a^{j_1+j_2},   \\
    \lim_{b\to a}  A_4  (a,b)
     =\, & -\frac{1}{4a^4}   \sum_{j_1,j_2,j_3} g_{j_1} g_{j_2} (g \star \delta \star \alpha \star (-1)^{\bullet} c \star
         (-1)^{\bullet} \beta)_{j_3}2^{j_3}    a^{j_1+j_2+j_3}  ,   \\
     \lim_{b\to a} A_5 (a,b) =\, & -\frac{1}{4a^4} 
     \sum_{j_1,j_2, j_3} g_{j_1} g_{j_2}   (\alpha \star g' \star \delta)_{j_3} (-1)^{j_1+j_2+j_3} 2^{j_3} a^{j_1+j_2+j_3} ,  \\
    \text{and}  \quad\quad\quad\quad\quad\quad
     \lim_{b\to a} A_6(a,b)  =\, &  -\frac{\mathcal{L}'}{2a^3}   \sum_{j_1,j_2, j_3} g_{j_1} g_{j_2} (g \star \alpha \star \delta)_{j_3} (-1)^{j_1+j_2+j_3}2^{j_3}  a^{j_1+j_2+j_3}.   
    \end{split}
    \end{equation}
It remains to compute $\lim_{b\to a}A_3(a,b)$. We have
    \begin{equation*}
    \begin{split}\label{simplifiedA3}
     A_3 (a,b)
     =\, & \frac{1}{ab^2(b-a)} \sum_{j_1, j_2}  
          g_{j_1} (c \star \beta \star (-1)^{\bullet} g)_{j_2}
            (-1)^{j_1+j_2}
          a^{j_1} b^{j_2} \\
          &\hspace{6em}\times \Big( (g \star \alpha \star \delta)_{0}+ (g \star \alpha \star \delta)_{1}(b-a)+\sum_{j_3\geq2}(g \star \alpha \star \delta)_{j_3} (b-a)^{j_3} \Big)\\&-\frac{1}{a^2b(b-a)} \sum_{j_1, j_2}  
          g_{j_1} (c \star \beta \star (-1)^{\bullet} g)_{j_2}
             (-1)^{j_1+j_2}
             a^{j_2} b^{j_1} \\
             &\hspace{6em} \times \Big((g \star \alpha \star \delta)_{0}-(g \star \alpha \star \delta)_{1}(b-a)+\sum_{j_3\geq2}(g \star \alpha \star \delta)_{j_3}(-1)^{j_3}  (b-a)^{j_3}\Big).
    \end{split}
    \end{equation*}
It follows that 
    \begin{equation}
    \begin{split} \label{limA3}
   &\lim_{b\to a} A_3(a,b) \\
     =\, & (g \star \alpha \star \delta)_{0}
     \lim_{b\to a}\bigg\{\frac{a}{a^2b^2(b-a)} \sum_{j_1}  
      g_{j_1} (-1)^{j_1}a^{j_1}\sum_{j_2}(c \star \beta \star (-1)^{\bullet} g)_{j_2}
        (-1)^{j_2}
      b^{j_2}\\
      &\hspace{8em} -  \frac{b}{a^2b^2(b-a)} \sum_{j_1}  
      g_{j_1} (-1)^{j_1}b^{j_1}\sum_{j_2}(c \star \beta \star (-1)^{\bullet} g)_{j_2}
         (-1)^{j_2}
         a^{j_2}  \bigg\}\\
      &+ \frac{(g \star \alpha \star \delta)_{1}}{a^3}\sum_{j_1}g_{j_1}(-1)^{j_1}a^{j_1}\sum_{j_2}(c \star \beta \star (-1)^{\bullet} g)_{j_2}(-1)^{j_2}a^{j_2}.
    \end{split}
    \end{equation}
At this point, we need the following lemma to simplify the limit on the right-hand side of \eqref{limA3}. 

\begin{lemma}
    Let $f_1$ and $f_2$ be entire functions. Consider 
        \[
        F(z_1,z_2) := \frac{f_1(z_1)f_2(z_2)-f_1(z_2) f_2(z_1)}{z_1-z_2}.
        \]
    Then 
        \[
        \lim_{b\to a}F(a,b)
         =   f_1'(a) f_2(a) -f_1(a) f_2'(a).
         \]
 \end{lemma} 
 
\begin{proof}
     Note that if $a \neq b$, then
        \[
        F(a,b) 
         =   \frac{(f_1(a)-f_1(b))f_2(a)}{a-b}
        -  \frac{f_1(a)(f_2(a)- f_2(b))}{a-b}.
        \]
    As $b \to a$, we  obtain $F(a,b)\to    f_1'(a) f_2(a) -f_1(a) f_2'(a)$. 
\end{proof}

We may apply this lemma for
   \[
   f_{1}(z)
   =z\sum_{j_1}g_{j_1}(-1)^{j_1}z^{j_1}\quad \text{and}\quad f_{2}(z)=\sum_{j_2}(c \star \beta \star (-1)^{\bullet} g)_{j_2}  (-1)^{j_2}z^{j_2}.
    \]
 These are both entire functions since $f_1(z)=-z^2\zeta(1-z)$ and $f_2(z)=zG(-z)e^{-zY}\zeta(1+z)$. When we apply the lemma in this setting, \eqref{limA3} becomes
    \begin{equation*}
    \begin{split}
    \lim_{b\to a} A_3 =& \,
    (g \star \alpha \star \delta)_{0} \lim_{b\to a}\bigg(\frac{1}{a^2b^2(b-a)} f_{1}(a)f_2(b)-  \frac{1}{a^2b^2(b-a)} f_1(b)f_2(a) \bigg)\\
    &+ \frac{(g \star \alpha \star \delta)_{1}}{a^4}f_1(a)f_{2}(a)
    \\
    =&\, \frac{(g \star \alpha \star \delta)_{0} }{a^4}\Big( f_1'(a) f_2(a) -f_1(a) f_2'(a)\Big)+ \frac{(g \star \alpha \star \delta)_{1}}{a^4}f_1(a)f_{2}(a)\\
    =&\, \frac{(g \star \alpha \star \delta)_{0} }{a^4}\sum_{j_1,j_2}g_{j_1}(c \star \beta \star (-1)^{\bullet} g)_{j_2}(j_2-j_1-1)(-1)^{j_1+j_2}a^{j_1+j_2}\\
    &+\frac{2(g \star \alpha \star \delta)_{1} }{a^3}\sum_{j_1,j_2}g_{j_1}(c \star \beta \star (-1)^{\bullet} g)_{j_2}(-1)^{j_1+j_2}a^{j_1+j_2}.
      \end{split}
      \end{equation*}
Then upon adding the right-hand sides of \eqref{eq:A1-A6} to the right-hand side of the last equation above, we obtain 
\begin{equation*}
\begin{split}
\mathcal{R}(a,a)&=\lim_{b\to a}\mathcal{R}(a, b)\\&=\frac{1}{a^4}\bigg(\sum_{j_1,j_2,j_3}C_1(j_1,j_2,j_3)a^{j_1+j_2+j_3}+\sum_{j_1,j_2}C_2(j_1,j_2)a^{j_1+j_2}\bigg)\\&\hspace{2em}+\frac{1}{a^3}\bigg(\sum_{j_1,j_2,j_3}D_1(j_1,j_2,j_3)a^{j_1+j_2+j_3}+\sum_{j_1,j_2}D_2(j_1,j_2)a^{j_1+j_2}\bigg),
\end{split}
\end{equation*}
where 
\begin{align*}
C_1(j_1,j_2,j_3)=&\, \frac{1}{2} (j_1+j_2-2)2^{j_3}
      g_{j_1} g_{j_2} (g \star \delta)_{j_3} 
     -\frac{1}{4} 2^{j_3}  g_{j_1} g_{j_2} (g \star \delta \star \alpha \star (-1)^{\bullet} c \star
         (-1)^{\bullet} \beta)_{j_3} \\
         &  -\frac{1}{4} (-1)^{j_1+j_2+j_3} 2^{j_3}g_{j_1} g_{j_2}   (\alpha \star g' \star \delta)_{j_3}  ,
\\
C_2(j_1,j_2) 
=
&\,  2\delta_0  (-1)^{j_1} 
      (\alpha*g*g)_{j_1} (g*g)_{j_2} 
      \\
   &+(-1)^{j_1+j_2}(j_2-j_1-1) (g \star \alpha \star \delta)_{0} g_{j_1}(c \star \beta \star (-1)^{\bullet} g)_{j_2},
\\
D_1(j_1,j_2,j_3)
=&\,
\frac{1}{2}(L+2g_1)  2^{j_3}
     g_{j_1} g_{j_2} (g \star \delta)_{j_3} + 2^{j_3}g_{j_1} g_{j_2} (g \star \delta')_{j_3} 
     \\ 
     &-\frac{\mathcal{L}'}{2} (-1)^{j_1+j_2+j_3}2^{j_3} g_{j_1} g_{j_2} (g \star \alpha \star \delta)_{j_3} 
         \end{align*}
 and
  \begin{align*}
 D_2(j_1,j_2)&=2(-1)^{j_1+j_2}(g \star \alpha \star \delta)_{1} g_{j_1}(c \star \beta \star (-1)^{\bullet} g)_{j_2}.
 \end{align*}
Hence
\begin{align*}
\mathcal{R}(a,a)&=\frac{1}{a^4} \sum_{j=0}^{\infty}C(j)a^{j},
\end{align*}
where \[C(0)=C_1(0,0,0)+C_2(0,0),\] and  for $i\in\mathbb{N}$ we have
\begin{align*}
C(j)
=&\,\sum_{\substack{j_1,j_2,j_3\\j_1+j_2+j_3=j-1}}D_1(j_1,j_2,j_3)+\sum_{\substack{j_1,j_2\\j_1+j_2=j-1}}D_2(j_1,j_2)
\\
&+\sum_{\substack{j_1,j_2,j_3\\j_1+j_2+j_3=j}}C_1(j_1,j_2,j_3)+\sum_{\substack{j_1,j_2\\j_1+j_2=j}}C_2(j_1,j_2).
\end{align*}
 Using Maple we show that $C(j)=0$ for $j=0,1,2,3$, and we compute
\begin{equation*}
\begin{split}
 \label{R00}
 \mathcal{R}(0,0) =
C(4)
=
\, & -\frac{7}{12}\delta_0 L^{4}-\delta_{0} L^{2} Y^{2} +\frac{4}{3} \delta_{0} L^{3} Y+\frac{1}{3} \delta_{0} L Y^{3}-\frac{1}{24} \delta_{0} Y^{4}\\
&+(-2 \delta_1 +\frac{4}{3} \delta_{0}c_1 -4 g_{1}\delta_0)L^3  +(\frac{1}{3} \delta_{1} -\frac{1}{6} \delta_{0}c_1  +\frac{2}{3} g_{1} \delta_{0})Y^3\\
&+(-2 \delta_{0} c_{1} +4\delta_1  +8 g_{1}\delta_0)L^2Y+(-2\delta_1 -4 g_{1}\delta_0 +\delta_{0}c_1)LY^2\\
&+(8 g_{1} c_{1} \delta_{0}-8 g_{1} \delta_{1}-4 \delta_2 -5 g_{1}^{2} \delta_0+4 c_1\delta_1 -2 \delta_{0}c_2  -6 g_2\delta_{0}) L^{2}\\
&+(-8 g_{1} c_{1} \delta_{0}+12 g_{1}^{2} \delta_{0} -4c_{1} \delta_{1} +2 \delta_{0} c_2+8 \delta_{0} g_{2} +16 g_{1}\delta_1+8  \delta_{2}) L Y\\
& +(\delta_{1}c_1  +2 g_{1}c_1 \delta_{0} -2 \delta_{2}-3 g_{1}^{2} \delta_{0} -\frac{1}{2} \delta_{0}  c_{2}-2 \delta_{0} g_{2}-4 g_{1} \delta_1)Y^{2} \\
&+(12 g_{1}^{2} \delta_{0}c_1    +8 c_{1} \delta_{0} g_{2} +16 g_{1} c_{1} \delta_{1}-8 g_{1} c_{2} \delta_{0} -4  \delta_{0} g_{1} g_{2} \\
&\hspace{2em} -8 \delta_{1} g_{2}-12 \delta_{0}  g_{3}-4 c_{2} \delta_{1}+2 \delta_{0}  c_{3} +8  \delta_{2} c_{1} +4 \delta_{0} g_{1}^{3}+4 g_{1}^{2} \delta_{1})L\\
&+(-\delta_{0} c_3    -6 g_{1}^{2} \delta_{0}c_1 -4 c_{1} \delta_{0} g_{2} -8 g_{1} c_{1} \delta_{1} +4 g_{1} \delta_{0} c_{2} +12\delta_{0} g_{1} g_{2} \\
&\hspace{2em} +8 \delta_{3}+12 g_{1}^{2} \delta_{1} +4  \delta_{0} g_{1}^{3}-4  c_{1} \delta_{2}+2 \delta_{1} c_{2}+4 \delta_{0} g_{3}+16 g_{1} \delta_{2}+8 \delta_{1} g_{2} )Y \\
&+16 \delta_{4}-16 \delta_{1} g_{3}+32 \delta_{3} g_{1}+32 g_{1}^{2} \delta_{2}-24 \delta_{0} g_{4}+8 g_{2}^{2} \delta_{0}+5 \delta_{0} g_{1}^{4}+16 \delta_{1} g_{1}^{3}\\
-&8 \delta_{0} g_{1} g_{3}+16 \delta_{1} g_{1} g_{2}+12 \delta_{0} g_{1}^{2} g_{2}+12 g_{1}^{2} \delta_{1} c_{1}+12 \delta_{0} g_{1} g_{2}c_1+8 \delta_{3}c_1+4 \delta_{0} g_{1}^{3}c_1
\\
&+4 \delta_{0} g_{3}c_1+8 \delta_{1} g_{2} c_{1}+16 g_{1} \delta_{2}c_1-4 \delta_{2}c_2-6 g_{1}^{2} \delta_{0} c_{2}-4 \delta_{0} g_{2}c_2-8 g_{1} \delta_{1}c_2\\
&+4 g_{1} \delta_{0} c_{3}+2 \delta_{1} c_{3}-\delta_{0} c_{4}.
\end{split}
\end{equation*}
Note that the above expression is a polynomial in $Y$ and $L$. 
By collecting terms of the same degree in $\mathbb{R}[Y,L]$ and noting from \eqref{eq:notation} that  $Y = \log K$ and $L = \log  \tfrac{t}{2 \pi} $, we find that 
\begin{equation*}
  \mathcal{R}(0,0) =   \sum_{j=0}^{4}  Q_j\left(Y, L  \right),
\end{equation*}
where the polynomials $Q_j$ are defined within the statement of Theorem \ref{thm22}.

\appendix

\section{Proof of Corollary \ref{cor}}\label{proof of cor}

Let $r(t) = \mathds{1}_{[T,2T]}(t)$ and choose smooth functions $\omega^{+}(t)$  and  $\omega^{-}(t)$ which satisfy 
    \[
        \omega^{-}(t) \le r(t) \le \omega^{+}(t),
    \]
where 
    \begin{equation*}
    \omega^{+}(t)=
    \begin{cases}
    0 \quad &\text{ if }\,\, \,  t < T-T_0 \,\,\, \text{ or } \,\,\, t>2T+T_0, \\
    1  \quad &\text{ if }\,\, \,  T+T_0 \leq t \leq 2T-T_0, 
    \end{cases}
     \end{equation*} 
and also
    \begin{equation*}
     (\omega^{\pm})^{(j)} \ll T_0^{-j}.
     \end{equation*} 
Note that 
    \begin{equation}
       \label{Dineq}
        \mathscr{D}_{2,2;  \omega^{-}}(K)  \le \mathscr{D}_{2,2;  r}(K)  \le  \mathscr{D}_{2,2;  \omega^{+}}(K), 
    \end{equation}
where we let
    \begin{equation*}
      \mathscr{D}_{2,2;  \omega^{\pm}}(K) 
     =        \sum_{j=0}^{4}  \int_{-\infty}^{\infty} \omega_{\pm}(t)  Q_j\left(\log K, \log \frac{t}{2 \pi} \right) \, dt 
                  + O\bigg(T^{\tfrac34(1+\eta)+\e}\Big(\frac{T}{T_0}\Big)^{\frac94}\bigg)+O\big(T^{1-\tfrac{\eta}{2}}\big) .
    \end{equation*}
It  follows from the above that 
    \begin{equation*}
    \begin{split}
      \label{diff}
       & \sum_{j=0}^{4}  \left\{ \int_{-\infty}^{\infty} \omega_{+}(t)  Q_j\Big(\log K, \log \frac{t}{2 \pi} \Big) \, dt 
       - \int_{-\infty}^{\infty} r(t)  Q_j\Big(\log K, \log \frac{t}{2 \pi} \Big) \, dt  \right\} 
       \\
      &  =  \sum_{j=0}^{4}  \left\{ \int_{T-T_0}^{T} + \int_{2T}^{2T+T_0} \right\} \omega_{+}(t)  Q_j\Big(\log K, \log \frac{t}{2 \pi} \Big) \, dt  
       \ll T_0 (\log T)^4. 
    \end{split}
    \end{equation*}
Note that a similar argument establishes the same bound when $\omega^{+}$ is replaced by $\omega^{-}$. Thus by \eqref{Dineq} we have
    \begin{align*}
      \mathscr{D}_{2,2;  r}(K)= &\,
        \sum_{j=0}^{4}  \int_{-\infty}^{\infty} r(t)  Q_j\Big(\log K, \log \frac{t}{2 \pi} \Big) \, dt  \\
                  &+ O\bigg(T^{\tfrac34(1+\eta)+\e}\Big(\frac{T}{T_0}\Big)^{\tfrac94}\bigg) +O\big(T^{1-\tfrac{\eta}{2}}\big)+ O\big(T_0 (\log T)^4 \big). 
    \end{align*}
We then select $T_0 =T^{\tfrac{12+3\eta}{13}}$ so that the first and the third error terms are equal, and obtain
   \begin{equation*}
      \mathscr{D}_{2,2;  r}(K)= 
        \sum_{j=0}^{4}  \int_{-\infty}^{\infty} r(t)  Q_j\Big(\log K, \log \frac{t}{2 \pi} \Big) \, dt 
                +O\Big(T^{\max\left\{\tfrac{12+3\eta}{13} ,1-\tfrac{\eta}{2}\right\}}\Big). 
    \end{equation*}



\section{Computation of the coefficients in Theorem \ref{thm22} }\label{coeff-comp}

In this section, we rewrite the expressions for $Q_0(x,y), Q_1(x,y), Q_2(x,y)$, and $Q_3(x,y)$ that appear in Theorem \ref{thm22} by using the definitions of $g_j$ and $\delta_j$ in terms of $\gamma_{j-1}$ and $\zeta^{(j)}(2)$ as described in  \eqref{gj} and \eqref{deltaj}. Note that $c_0=1$ and the rest of the coefficients $c_j$ that appear in Theorem \ref{thm22} depend on the smoothing function $\varphi$.

Using Maple we compute the following expressions for $Q_3, Q_2, Q_1$ and $Q_0$.

\begin{align*}
Q_3(x,y)
&=\left(\frac{4 \gamma}{\pi^{2}}-\frac{12 \zeta^{\prime}\! \left(2\right)}{\pi^{4}}-\frac{c_1}{\pi^{2}}\right) x^{3}+\left(\frac{6 c_1}{\pi^{2}}-\frac{24 \gamma}{\pi^{2}}+\frac{72 \zeta^{\prime}\! \left(2\right)}{\pi^{4}}\right) x^{2} y \\&+\left(-\frac{12 c_1}{\pi^{2}}+\frac{48 \gamma}{\pi^{2}}-\frac{144 \zeta^{\prime}\! \left(2\right)}{\pi^{4}}\right) x \,y^{2}+\left(-\frac{24 \gamma}{\pi^{2}}+\frac{72 \zeta^{\prime}\! \left(2\right)}{\pi^{4}}+\frac{8 c_1}{\pi^{2}}\right) y^{3}.
\end{align*}

\begin{align*}
& Q_2(x,y) \\
&= \bigg(\frac{12 \gamma_1}{\pi^{2}}-\frac{18 \gamma^{2}}{\pi^{2}}+\frac{144 \zeta^{\prime}\! \left(2\right) \gamma}{\pi^{4}}-\frac{432 {\zeta^{\prime}\! \left(2\right)}^{2}}{\pi^{6}}+\frac{36 \zeta^{\prime\prime}\! \left(2\right)}{\pi^{4}}+\frac{12 c_1 \gamma}{\pi^{2}}-\frac{36 \zeta^{\prime}\! \left(2\right) c_1}{\pi^{4}}-\frac{3 c_2}{\pi^{2}}\bigg) x^{2}\\
&+\bigg(-\frac{48 c_1\gamma}{\pi^{2}}+\frac{72 \gamma^{2}}{\pi^{2}}+\frac{144 \zeta^{\prime}\! \left(2\right) c_1}{\pi^{4}}+\frac{12 c_2}{\pi^{2}}-\frac{48 \gamma_1}{\pi^{2}}-\frac{576 \zeta^{\prime}\! \left(2\right) \gamma}{\pi^{4}}+\frac{1728 {\zeta^{\prime}\! \left(2\right)}^{2}}{\pi^{6}}-\frac{144 \zeta^{\prime\prime}\! \left(2\right)}{\pi^{4}}\bigg) x y\\
& +\bigg(\frac{48 c_1 \gamma}{\pi^{2}}-\frac{30 \gamma^{2}}{\pi^{2}}-\frac{144 \zeta^{\prime}\! \left(2\right) c_1}{\pi^{4}}-\frac{12 c_2}{\pi^{2}}+\frac{36 \gamma_1}{\pi^{2}}+\frac{288 \zeta^{\prime}\! \left(2\right) \gamma}{\pi^{4}}-\frac{864 {\zeta^{\prime}\! \left(2\right)}^{2}}{\pi^{6}}+\frac{72 \zeta^{\prime\prime}\! \left(2\right)}{\pi^{4}}\bigg) y^{2}.
\end{align*}

\begin{align*}
&Q_1(x,y) \\
&=\bigg(-\frac{36 c_1 \,\gamma^{2}}{\pi^{2}}+\frac{24 \gamma^{3}}{\pi^{2}}+\frac{24 c_1\gamma_1}{\pi^{2}}+\frac{288 c_1 \zeta^{\prime}\! \left(2\right) \gamma}{\pi^{4}}+\frac{24 c_2 \gamma}{\pi^{2}}-\frac{72 \gamma  \gamma_1}{\pi^{2}}-\frac{432 \zeta^{\prime}\! \left(2\right) \gamma^{2}}{\pi^{4}}\\
&\quad \, -4 c_1 \Big(\frac{216 {\zeta^{\prime}\! \left(2\right)}^{2}}{\pi^{6}}-\frac{18 \zeta^{\prime\prime}\! \left(2\right)}{\pi^{4}}\Big)-\frac{72 c_2\zeta^{\prime}\! \left(2\right)}{\pi^{4}}-\frac{6 c_3}{\pi^{2}}+\frac{12 \gamma_2}{\pi^{2}}+\frac{288 \zeta^{\prime}\! \left(2\right) \gamma_1}{\pi^{4}}\\
&\quad \,+16 \Big(\frac{216 {\zeta^{\prime}\! \left(2\right)}^{2}}{\pi^{6}}-\frac{18 \zeta^{\prime\prime}\! \left(2\right)}{\pi^{4}}\Big) \gamma -\frac{10368 {\zeta^{\prime}\! \left(2\right)}^{3}}{\pi^{8}}+\frac{1728 \zeta^{\prime}\! \left(2\right) \zeta^{\prime\prime}\! \left(2\right)}{\pi^{6}}-\frac{48 \zeta^{\prime\prime\prime}\! \left(2\right)}{\pi^{4}}\bigg) x\\
& +\bigg(\frac{72 c_1 \,\gamma^{2}}{\pi^{2}}+\frac{24 \gamma^{3}}{\pi^{2}}-\frac{48 c_1\gamma_1}{\pi^{2}}-\frac{576 c_1 \zeta^{\prime}\! \left(2\right) \gamma}{\pi^{4}}-\frac{48 c_2 \gamma}{\pi^{2}}+\frac{24 \gamma  \gamma_1}{\pi^{2}}\\
&-\frac{144 \zeta^{\prime}\! \left(2\right) \gamma^{2}}{\pi^{4}}+8 c_1 \Big(\frac{216 {\zeta^{\prime}\! \left(2\right)}^{2}}{\pi^{6}}-\frac{18 \zeta^{\prime\prime}\! \left(2\right)}{\pi^{4}}\Big)+\frac{144 c_2 \zeta^{\prime}\! \left(2\right)}{\pi^{4}}+\frac{12 c_3}{\pi^{2}}-\frac{36 \gamma_2}{\pi^{2}}-\frac{288 \zeta^{\prime}\! \left(2\right) \gamma_1}{\pi^{4}}\bigg) y.
\end{align*}

\begin{align*}
&Q_0(x,y) \\
&= -\frac{31104 {\zeta^{\prime}\! \left(2\right)}^{2} \zeta^{\prime\prime}\! \left(2\right)}{\pi^{8}}+\frac{1152 \zeta^{\prime}\! \left(2\right) \zeta^{\prime\prime\prime}\! \left(2\right)}{\pi^{6}}-\frac{72 \zeta^{\prime}\! \left(2\right)c_3}{\pi^{4}}
-\frac{6 c_4}{\pi^{2}}\\
&+8 \bigg(-\frac{1296 {\zeta^{\prime}\! \left(2\right)}^{3}}{\pi^{8}}+\frac{216 \zeta^{\prime}\! \left(2\right) \zeta^{\prime\prime}\! \left(2\right)}{\pi^{6}}-\frac{6 \zeta^{\prime\prime\prime}\! \left(2\right)}{\pi^{4}}\bigg) c_1 -4 \bigg(\frac{216 {\zeta^{\prime}\! \left(2\right)}^{2}}{\pi^{6}}-\frac{18 \zeta^{\prime\prime}\! \left(2\right)}{\pi^{4}}\bigg) c_2\\
& -\frac{24 \zeta^{\left(4\right)}\! \left(2\right)}{\pi^{4}}+\frac{864 {\zeta^{\prime\prime}\! \left(2\right)}^{2}}{\pi^{6}}+\frac{124416 {\zeta^{\prime}\! \left(2\right)}^{4}}{\pi^{10}}+\frac{24 \gamma_3}{\pi^{2}}+\frac{30 \gamma^{4}}{\pi^{2}}+\frac{48 \gamma_1^{2}}{\pi^{2}}+\frac{24 \gamma_1c_2}{\pi^{2}}\\
& -\frac{36 \gamma^{2} c_2}{\pi^{2}}-\frac{72 \gamma^{2} \gamma_1}{\pi^{2}}+32 \bigg(-\frac{1296 {\zeta^{\prime}\! \left(2\right)}^{3}}{\pi^{8}}+\frac{216 \zeta^{\prime}\! \left(2\right) \zeta^{\prime\prime}\! \left(2\right)}{\pi^{6}}-\frac{6 \zeta^{\prime\prime\prime}\! \left(2\right)}{\pi^{4}}\bigg) \gamma \\
&+32 \gamma^{2} \bigg(\frac{216 {\zeta^{\prime}\! \left(2\right)}^{2}}{\pi^{6}}-\frac{18 \zeta^{\prime\prime}\! \left(2\right)}{\pi^{4}}\bigg)+16 \gamma  \bigg(\frac{216 {\zeta^{\prime}\! \left(2\right)}^{2}}{\pi^{6}}-\frac{18 \zeta^{\prime\prime}\! \left(2\right)}{\pi^{4}}\bigg) c_1 \\
&+\frac{288 \zeta^{\prime}\! \left(2\right) \gamma_2}{\pi^{4}}+\frac{12 \gamma_2c_1}{\pi^{2}}+\frac{24 \gamma  c_3}{\pi^{2}}+\frac{24 \gamma^{3} c_1}{\pi^{2}}-\frac{24 \gamma  \gamma_2}{\pi^{2}}-\frac{576 \zeta^{\prime}\! \left(2\right) \gamma^{3}}{\pi^{4}}\\
&+\frac{288 \zeta^{\prime}\! \left(2\right) \gamma_1 c_1}{\pi^{4}}+\frac{288 \gamma  \zeta^{\prime}\! \left(2\right) c_2}{\pi^{4}}+\frac{576 \zeta^{\prime}\! \left(2\right) \gamma  \gamma_1}{\pi^{4}}-\frac{72 \gamma  \gamma_1 c_1}{\pi^{2}}-\frac{432 \gamma^{2} \zeta^{\prime}\! \left(2\right) c_1}{\pi^{4}}.
\end{align*}


\end{document}